\documentclass[11 pt]{amsart}
\usepackage{latexsym,amscd,amssymb, graphicx, amsthm, mathtools, young}  
\usepackage{blkarray}
\usepackage{tikz-cd, xcolor}

\usepackage[margin=1in]{geometry}

\numberwithin{equation}{section}

\newtheorem{theorem}{Theorem}[section]

\newtheorem{corollary}[theorem]{Corollary}
\newtheorem{lemma}[theorem]{Lemma}
\newtheorem{conjecture}[theorem]{Conjecture}
\newtheorem{observation}[theorem]{Observation}

\newtheorem{remark}[theorem]{Remark}

\definecolor{2purple}{RGB}{204,102,255}
\definecolor{3green}{RGB}{0,204,0}

\newtheorem{defn}[theorem]{Definition}
\theoremstyle{definition}

\newcommand{\II}{{\mathbf {I}}}
\newcommand{\gr}{{\mathrm {gr}}}

\newcommand{\lis}{{\mathrm {lis}}}
\newcommand{\exc}{{\mathrm {exc}}}
\newcommand{\inv}{{\mathrm {inv}}}
\newcommand{\peak}{{\mathrm {peak}}}
\newcommand{\triv}{{\mathrm {triv}}}

\newcommand{\SYT}{{\mathrm {SYT}}}
\newcommand{\Loc}{{\mathrm {Loc}}}
\newcommand{\End}{{\mathrm {End}}}
\newcommand{\spn}{{\mathrm {span}}}

\newcommand{\Res}{{\mathrm {Res}}}
\newcommand{\Class}{{\mathrm {Class}}}
\newcommand{\sign}{{\mathrm {sign}}}

\newcommand{\Hilb}{{\mathrm {Hilb}}}
\newcommand{\symm}{{\mathfrak {S}}}

\newcommand{\sss}{{\mathfrak {s}}}
\newcommand{\SSS}{{\mathcal {S}}}
\newcommand{\BBB}{{\mathcal {B}}}
\newcommand{\RRR}{{\mathcal {R}}}

\newcommand{\FFF}{{\mathbb {F}}}
\newcommand{\CC}{{\mathbb {C}}}
\newcommand{\ZZ}{{\mathbb {Z}}}
\newcommand{\QQ}{{\mathbb {Q}}}

\newcommand{\initial}{{\mathrm {in}}}

\newcommand{\xx}{{\mathbf {x}}}
\newcommand{\vv}{{\mathbf {v}}}
\newcommand{\AAA}{{\mathbf {A}}}

\newcommand{\III}{{\mathcal {I}}}

\newcommand{\one}{{\mathbf {1}}}
\newcommand{\zero}{{\mathbf {0}}}

\newcommand{\zz}{{\mathbf {z}}}
\newcommand{\EE}{{\mathbf {E}}}

\newcommand{\Top}{{\mathrm {Top}}}

\begin{document}

\title[Increasing subsequences, matrix loci, and Viennot shadows]
{Increasing subsequences, matrix loci, and Viennot shadows}



\author{Brendon Rhoades}
\address
{Department of Mathematics \newline \indent
University of California, San Diego \newline \indent
La Jolla, CA, 92093-0112, USA}
\email{bprhoades@ucsd.edu}

\begin{abstract}
Let $\xx_{n \times n}$ be an $n \times n$ matrix of variables and let $\FFF[\xx_{n \times n}]$ be the polynomial ring in these variables over a field $\FFF$.
We study the ideal $I_n \subseteq \FFF[\xx_{n \times n}]$ generated by all row and column variable sums and all products of two variables drawn from the same 
row or column.  We show that the quotient $\FFF[\xx_{n \times n}]/I_n$ admits a standard monomial basis determined by Viennot's shadow line avatar of the Schensted 
correspondence. As a corollary, the Hilbert series of $\FFF[\xx_{n \times n}]/I_n$ is the generating function of permutations in $\symm_n$ by 
the length of their longest increasing subsequence.
Along the way, we describe a `shadow junta' basis of the vector space of $k$-local permutation statistics.
We also calculate the structure of $\FFF[\xx_{n \times n}]/I_n$ as a graded $\symm_n \times \symm_n$-module.
\end{abstract}

\maketitle

\section{Introduction}
\label{Introduction}

Let $\xx$ be a finite set of variables and let $\FFF[\xx]$ be the polynomial ring in these variables over 
a field $\FFF$. If $I \subseteq \FFF[\xx]$ is a homogeneous ideal, the quotient ring $\FFF[\xx]/I$ has the structure of a graded vector space.
The {\em Hilbert series} of $\FFF[\xx]/I$ is the graded dimension of the vector space, viz.
\begin{equation}
\Hilb( \FFF[\xx]/I; q) := \sum_{d \geq 0}  \dim_{\FFF} (\FFF[\xx]/I)_d \cdot q^d.
\end{equation}
Macaulay \cite{Macaulay} characterized the polynomials
$a_0 + a_1 \cdot q + \cdots + a_d \cdot q^d$ with positive integer coefficients which arise as the Hilbert series of a graded quotient of the form $\FFF[\xx]/I$. 
Following the exposition of Stanley \cite[Thm. 1.3]{Stanley}, for positive integers $a$ and $i$, there is a unique representation $a = {b_i \choose i} + {b_{i-1} \choose i-1} + \cdots + {b_j \choose j}$ where $b_i > b_{i-1} > \cdots > b_j \geq j \geq 1$. Let $a^{\langle i \rangle} := {b_i + 1 \choose i+1} + \cdots + {b_j + 1 \choose j +1}$. Then $a_0 + a_1 \cdot q + \cdots + a_d \cdot q^d$ is the Hilbert series of some graded quotient $\FFF[\xx]/I$ if and only if $a_{i+1} \leq a_i^{\langle i+1 \rangle}$ for $0 \leq i \leq d-1$.

In this paper we show that a generating function arising from increasing subsequences of permutations is the Hilbert series of a natural graded ring.
Write $\symm_n$ for the symmetric group on $[n] := \{1, \dots, n \}$. If $w \in \symm_n$ is a permutation, an {\em increasing subsequence} in $w$ is a set 
of positions $1 \leq i_1 < \cdots < i_k \leq n$ whose images under $w$ satisfy $w(i_1) < \cdots < w(i_k)$. The integer $k$ is the {\em length} of this increasing subsequence.
We write
\begin{equation}
\lis(w) := \max \{ k \,:\, \text{$w$ has an increasing subsequence of length $k$} \}
\end{equation}
for the length of the longest increasing subsequence of $w$ and
\begin{equation}
a_{n,k} := | \{ w \in \symm_n \,:\, \lis(w) = k \} |
\end{equation}
for the number of permutations in $\symm_n$ whose longest increasing subsequence has length $k$.
For any positive integer $n$, the sequence $(a_{n,1}, a_{n,2}, \dots, a_{n,n})$ was conjectured by Chen \cite[Conj. 1.1]{Chen} to be log-concave, i.e.
$a_{n,i}^2 \geq a_{n,i-1} \cdot a_{n,i+1}$ for all $1 < i < n$. When $n = 4$, this sequence reads $(a_{4,1},a_{4,2},a_{4,3},a_{4,4}) = (1,13,9,1)$.

The following  ideal $I_n$ is our object of study.
Despite the simplicity of its generating set, it will turn out to have deep connections to the combinatorics of increasing subsequences.

\begin{defn}
\label{i-defn}
Let $\xx_{n \times n}$ be an $n \times n$ matrix of variables $(x_{i,j})_{1 \leq i , j \leq n}$ and consider the polynomial ring $\FFF[\xx_{n \times n}]$ over these variables.
Let $I_n \subseteq \FFF[\xx_{n \times n}]$ be the ideal generated by
\begin{itemize}
\item any product $x_{i,j} \cdot x_{i,j'}$ for $1 \leq i \leq n$ and $1 \leq j, j' \leq n$ of variables in the same row,
\item any product $x_{i,j} \cdot x_{i',j}$ for $1 \leq i, i' \leq n$ and $1 \leq j \leq n$ of variables in the same column,
\item any row sum $x_{i,1} + \cdots + x_{i,n}$ for $1 \leq i \leq n$, and
\item and column sum $x_{1,j} + \cdots + x_{n,j}$ for $1 \leq j \leq n$.
\end{itemize}
\end{defn}

The ideal $I_n \subseteq \FFF[\xx_{n \times n}]$ is homogeneous, so  $\FFF[\xx_{n \times n}]/I_n$ is a graded $\FFF$-algebra.
The natural action of the group $\symm_n \times \symm_n$ on the variable matrix $\xx_{n \times n}$ given by independent row and column permutation
induces an action on $\FFF[\xx_{n \times n}]$ which stabilizes $I_n$, so that $\FFF[\xx_{n \times n}]/I_n$ is a graded $\symm_n \times \symm_n$-module.

When $n = 1$ we have $I_1 = (x_{1,1}) \subseteq \FFF[\xx_{1 \times 1}]$ so that $\FFF[\xx_{1 \times 1}]/I_1 = \FFF$.
When $n = 2$ the ideal $I_2 \subseteq \FFF[\xx_{2 \times 2}]$ has generators
\begin{multline*}
x_{1,1}^2,  \, \,x_{1,2}^2, \, \, x_{2,1}^2,  \, \, x_{2,2}^2,  \, \, x_{1,1} x_{1,2}, \, \, x_{1,1} x_{2,1}, \, \, x_{1,2} x_{2,2}, \, \, x_{2,1} x_{2,2}, \\ x_{1,1} + x_{1,2},  \, \, x_{1,1} + x_{2,1}, \, \,
 x_{1,2} + x_{2,2}, \, \, x_{2,1} + x_{2,2}
\end{multline*}
and it is not hard to check that $\FFF[\xx_{2 \times 2}]/I_2$ has Hilbert series $1 + q$ and that the set of monomials $\{1, x_{1,2} \}$ descends to a basis.

We  prove (Corollary~\ref{hilbert-series}) that the Hilbert series of $\FFF[\xx_{n \times n}]/I_n$ is given by 
\begin{equation}
\Hilb( \FFF[\xx_{n \times n}] / I_n; q) = a_{n,n} + a_{n,n-1} \cdot q + a_{n,n-2} \cdot q^2  + \cdots + a_{n,1} \cdot q^{n-1}
\end{equation}
so that the (reversal of the) generating function for permutations in $\symm_n$ by longest increasing subsequence is the Hilbert series of 
$\FFF[\xx_{n \times n}] / I_n$. 
In particular, the polynomial $a_{n,n} + a_{n,n-1} \cdot q + a_{n,n-2} \cdot q^2  + \cdots + a_{n,1} \cdot q^{n-1}$ satisfies Macaulay's Criterion, a fact which seems difficult to prove directly from the combinatorics of increasing subsequences.
Taking $q \to 1$, the ungraded vector space $\FFF[\xx_{n \times n}] / I_n$ has dimension
\begin{equation}
\dim \FFF[\xx_{n \times n}] / I_n = n!.
\end{equation}
Furthermore, we will attach (Definition~\ref{shadow-defn})
 a monomial $\sss(w)$ in the variables $x_{i,j}$ to any permutation $w \in \symm_n$ such that 
\begin{equation}
\deg \sss(w) = n - \lis(w)
\end{equation}
and prove (Theorem~\ref{vector-space-structure}) that 
\begin{equation}
 \{ \sss(w) \,:\, w \in \symm_n \}
\end{equation}
descends to a vector space basis of $\FFF[\xx_{n \times n}] / I_n$.
In fact, this will be the standard monomial basis of  $\FFF[\xx_{n \times n}] / I_n$ with respect to a ``Toeplitz term order"  $<_\Top$
(Definition~\ref{toeplitz-defn}).
The notation $\sss$ refers to the use of Viennot's {\em shadow line} formulation \cite{Viennot} of the Schensted correspondence in the definition of $\sss(w)$.
Our results may be interpreted as the ideal $I_n \subseteq \FFF[\xx_{n \times n}]$ together with the term order $<_\Top$ ``seeing" the Viennot shadow line construction.

When the field $\FFF$ has characteristic zero or  characteristic $p > n$, we characterize the structure of $\FFF[\xx_{n \times n}]/I_n$ as an
ungraded (Corollary~\ref{ungraded-module-structure}) and graded (Theorem~\ref{graded-module-structure}) module over the product group
$\symm_n \times \symm_n$.
The module structure of $\FFF[\xx_{n \times n}]/I_n$ relates to a family of $\symm_n$-characters considered by Novak and the author \cite{NR}
in a strengthening of Chen's log-concavity conjecture.

For $1 \leq k \leq n$, define a character $\alpha_{n,k}: \symm_n \rightarrow \FFF$ by the rule
\begin{equation}
\alpha_{n,k} := \sum_{\substack{\lambda \, \vdash \, n \\ \lambda_1 \, = \, k}} f^{\lambda} \cdot \chi^{\lambda}
\end{equation}
where the sum is over partitions of $n$ whose first row has length $k$. Here $\chi^{\lambda}: \symm_n \rightarrow \FFF$ is the irreducible character
of $\symm_n$ attached to the partition $\lambda$ and $f^\lambda = \chi^\lambda(e)$ is the dimension of the irreducible $\symm_n$-module attached to $\lambda$.
We have $\alpha_{n,k}(e) = a_{n,k}$, so the sequence $(\alpha_{n,1}, \dots, \alpha_{n,n})$ of class functions is a representation-theoretic refinement of the sequence
$(a_{n,1}, \dots, a_{n,n})$ appearing in Chen's conjecture.

Novak and the author conjectured \cite[Conj. 2]{NR} the the difference $\alpha_{n,k} * \alpha_{n,k} - \alpha_{n,k-1} * \alpha_{n,k+1}$ is 
a genuine (rather than merely virtual) character of $\symm_n$ for all $1 < k < n$, where $*$ denotes the {\em Kronecker product} of class functions on $\symm_n$.
Since $\alpha_{n,k}(e) = a_{n,k}$, this would imply Chen's conjecture.
One way to prove this stronger conjecture would be to describe an $\symm_n$-module which has $\alpha_{n,k} * \alpha_{n,k} - \alpha_{n,k-1} * \alpha_{n,k+1}$ 
as its character.
We prove (Corollary~\ref{novak-rhoades-corollary}) that $\alpha_{n,k}$ is the character of the degree $n-k$ piece of the 
quotient  $\FFF[\xx_{n \times n}]/I_n$, restricted from the product  $\symm_n \times \symm_n$ to either factor of $\symm_n$.
To the author's knowledge, this is the simplest explicit module with character $\alpha_{n,k}$.
We hope that this representation-theoretic model for $\alpha_{n,k}$ can give new insight on the Novak-Rhoades conjecture.
In fact, it appears that a stronger equivariant log-concavity result holds without restriction from $\symm_n \times \symm_n$ to one of its factors;
see Conjecture~\ref{equivariant-conjecture}.

Our results have application to permutation statistics.
For $k \geq 0$, a statistic $f: \symm_n \rightarrow \FFF$ 
is {\em $k$-local} \cite{DFLLV, HR} if $f$ is an $\FFF$-linear combination
of indicator statistics which detect whether a permutation $w$ carries a given list of $k$ positions onto another given list of $k$ values.
The locality of a permutation statistic is a measure of its complexity; for example, the 0-local statistics are precisely the constant functions $\symm_n \rightarrow \FFF$.
While the vector space of $k$-local statistics is defined via a spanning set, finding an explicit basis for this vector space was an open problem in \cite{HR}.
Our Gr\"obner-theoretic methods  yield (Theorem~\ref{local-basis-theorem}) a solution to this problem.

To prove our results, we apply the method of orbit harmonics to the locus $P_n \subseteq \FFF^{n \times n}$ of permutation matrices inside the affine space
$\FFF^{n \times n}$ of $n \times n$ matrices over $\FFF$.
Orbit harmonics is a general method of transforming finite point loci $Z \subseteq \FFF^N$ into graded quotients $\FFF[\xx_N]/\gr \, \II(Z)$
of the polynomial ring $\FFF[\xx_N]$.
This method dates back to at least the work of Kostant \cite{Kostant} 
and has been used to study modules arising in Macdonald theory \cite{GP, Griffin, HRS},
understand cyclic sieving results \cite{OR}, 
and interpret Donaldson-Thomas invariants of symmetric quivers as orbit enumerations in the lattice points of break divisor polytopes \cite{RRT}.

The rest of the paper is organized as follows. 
In {\bf Section~\ref{Background}} we give background material on Gr\"obner bases, orbit harmonics, and the Schensted correspondence.
In {\bf Section~\ref{Hilbert}} we use Viennot's shadow line interpretation of the Schensted correspondence to find a monomial
basis of $\FFF[\xx_{n \times n}]/I_n$ indexed by permutations in $\symm_n$.
We also give a basis for the space of $k$-local permutation statistics.
In {\bf Section~\ref{Module}} we describe the structure of $\FFF[\xx_{n \times n}]/I_n$ as a module over the product group $\symm_n \times \symm_n$.
We close in {\bf Section~\ref{Conclusion}} with directions for future research.

\section{Background}
\label{Background}

\subsection{Gr\"obner theory} Let $\xx = (x_1, \dots, x_N)$ be a finite list of variables and let $\FFF[\xx_N]$ be the polynomial ring in these variables over a field $\FFF$.
A total order $<$ on the monomials in $\FFF[\xx_N]$ is a {\em term order} if 
\begin{itemize}
\item we have $1 \leq m$ for all monomials $m$, and
\item if $m_1, m_2, m_3$ are monomials with $m_1 \leq m_2$, then $m_1 m_3 \leq m_2 m_3$.
\end{itemize}
If $f \in \FFF[\xx_N]$ is a nonzero polynomial and $<$ is a term order, write $\initial_< (f)$ for the largest monomial with respect to $<$ which appears with nonzero coefficient in $f$.

Let $I \subseteq \FFF[\xx_N]$ be an ideal and let $<$ be a term order. 
The {\em initial ideal} $\initial_< (I) \subseteq \FFF[\xx_N]$ associated to $I$ is given by
\begin{equation}
\initial_< (I) := \langle \initial_< (f) \,:\, f \in I, \, \, f \neq 0 \rangle \subseteq \FFF[\xx_N].
\end{equation}
In other words, the ideal $\initial_< (I)$ is generated by the $<$-leading monomials of all nonzero polynomials in $I$.  A subset $G = \{ g_1, \dots, g_r \} \subseteq I$ 
is a {\em Gr\"obner basis} of $I$ if
\begin{equation}
\initial_< (I) = \langle \initial_< (g_1), \dots, \initial_< (g_r) \rangle.
\end{equation}
If $G = \{ g_1, \dots, g_r \}$ is a Gr\"obner basis of $I$ it follows that $I = \langle g_1, \dots, g_r \rangle$.

Given an ideal $I \subseteq \FFF[\xx_N]$ and a term order $<$, a monomial $m$ in the variables $\xx_N$ is a {\em standard monomial} if 
$m \neq \initial_< (f)$ for any nonzero $f \in I$. It is known that the family of cosets 
\begin{equation}
 \{ m + I \,:\, \text{$m$ a standard monomial} \}
\end{equation}
descends to a vector space basis of $\FFF[\xx_N]/I$.  This is referred to as the {\em standard monomial basis}.

\subsection{Orbit harmonics}
Let $Z \subseteq \FFF^N$ be a finite locus of points and consider the ideal
\begin{equation}
\II(Z) := \{ f \in \FFF[\xx_N] \,:\, f(\zz) = 0 \text{ for all } \zz \in Z \}
\end{equation}
of polynomials in $\FFF[\xx_N]$ which vanish on $Z$. 
The ideal $\II(Z)$ is usually not homogeneous.
Since $Z$ is finite, we have an identification
\begin{equation}
\label{short-chain}
\FFF[Z] \cong \FFF[\xx_N]/\II(Z)
\end{equation}
of the vector space $\FFF[Z]$ of functions $Z \rightarrow \FFF$ and the typically ungraded
quotient space $\FFF[\xx_N]/\II(Z)$.

Given a nonzero polynomial $f \in \FFF[\xx_N]$, let $\tau(f)$ be the highest degree homogeneous component of $f$.
That is, if $f = f_d + \cdots +  f_1 + f_0$ where $f_i$ is homogeneous of degree $i$ and $f_d \neq 0$, we have $\tau(f) = f_d$.
If $I \subseteq \FFF[\xx_N]$ is an ideal, the {\em associated graded ideal} is
\begin{equation}
\gr \, I := \langle \tau(f) \,:\, f \in I, \, \, f \neq 0 \rangle.
\end{equation}
In other words, the ideal $\gr \, I$ is generated by the top homogeneous components of all nonzero polynomials in $I$.
The associated graded ideal $\gr \, I \subseteq \FFF[\xx_N]$ is homogeneous by construction.

Returning to the setting of our finite locus $Z \subseteq \FFF^N$, we may extend the chain \eqref{short-chain} of ungraded $\FFF$-vector space
isomorphisms
\begin{equation}
\label{long-chain}
\FFF[Z] \cong \FFF[\xx_N] / \II(Z) \cong \FFF[\xx_N] / \gr \, \II(Z)
\end{equation}
where the last quotient $\FFF[\xx_N] / \gr \, \II(Z)$ has the additional structure of a graded $\FFF$-vector space.

When the locus $Z$ possesses symmetry, more can be said.
Let $G \subseteq GL_N(\FFF)$ be a finite matrix group and assume that the group algebra $\FFF[G]$ is semisimple.
Equivalently, this means that $|G| \neq 0$ in $\FFF$.
The natural action of $G$ on $\FFF^N$ induces an action on $\FFF[\xx_N]$ by linear substitutions.
If $Z$ is stable under the action of $G$, the isomorphisms \eqref{long-chain} hold in the category of $\FFF[G]$-modules, 
and the last quotient $\FFF[\xx_N] / \gr \, \II(Z)$ has the additional structure of a graded $\FFF[G]$-module.

\subsection{The Schensted correspondence}
Given $n \geq 0$, a {\em partition of $n$} is a weakly decreasing sequence $\lambda = (\lambda_1 \geq \cdots \geq \lambda_k)$ of positive integers
which satisfy $\lambda_1 + \cdots + \lambda_k = n$.
We write $\lambda \vdash n$ to indicate that $\lambda$ is a partition of $n$. We identify a partition $\lambda = (\lambda_1, \dots, \lambda_k)$ with its
(English) {\em Young diagram} consisting of $\lambda_i$ left-justified boxes in row $i$.

Let $\lambda \vdash n$ be a partition.
A {\em tableau} of shape $\lambda$ is an assignment $T: \lambda \rightarrow \{1,2,\dots\}$ of positive integers to the boxes of $\lambda$.
 A {\em standard tableau} of shape $\lambda$ is a bijective filling 
$T: \lambda \rightarrow [n]$ of the boxes of $\lambda$ with $1, 2, \dots, n$ which is increasing across rows and down columns.
We display, from left to right, the Young diagram of $\lambda = (4,2,2) \vdash 8$, a standard tableau of shape $\lambda$, and another tableau of 
shape $\lambda$ which is not standard.
\begin{center}
\begin{Young}
 & & & \cr
 & \cr
 & \cr
\end{Young}   \quad \quad
\begin{Young}
 1 & 2  & 5 & 8 \cr
 3 & 4 \cr
 6 & 7 \cr
\end{Young}   \quad \quad
\begin{Young}
 1 & 3 & 7 & 12 \cr
 4 & 5 \cr
 9 & 10 \cr
\end{Young}   
\end{center}
Although the above tableau $T: \lambda \rightarrow \{1,2, \dots \}$ on the far right is not standard, it is an injective filling which is (strictly)
increasing across rows and down columns.  We call a tableau satisfying these conditions a {\em partial standard tableau}.

The famous {\em Schensted correspondence} \cite{Schensted} is a bijection
\begin{equation}
\symm_n \xrightarrow{ \quad \sim \quad} \bigsqcup_{\lambda \vdash n} \{ (P, Q) \,:\, P, Q \in \SYT(\lambda) \}
\end{equation}
which sends a permutation $w \in \symm_n$ to a pair $(P(w), Q(w))$ of standard  tableaux with the same $n$-box shape.
The Schensted correspondence is most commonly defined using an insertion algorithm (see e.g. \cite{Sagan} for details).
We will not need the insertion formulation of the Schensted bijection, but 
 an equivalent ``geometric" formulation due to Viennot \cite{Viennot} recalled in the next section
will be crucial in our work.
Schensted proved that his bijection relates to increasing subsequences as follows.

\begin{theorem}
\label{schensted-theorem} (Schensted \cite[Thm. 1]{Schensted})
Let $w \in \symm_n$ and suppose that $w \mapsto (P(w), Q(w))$ under the Schensted bijection where $P(w)$ and $Q(w)$ have shape $\lambda \vdash n$.
The first part $\lambda_1$ of the partition $\lambda$ is the length of the longest increasing subsequence of $w$.
\end{theorem}


\subsection{$\symm_n$-representation theory}
Let $\FFF$ be a field in which $n \neq 0$ so that the group algebra 
$\FFF[\symm_n]$ is semisimple.
There is a one-to-one correspondence between partitions of $n$ and irreducible representations of $\symm_n$ over $\FFF$.
If $\lambda \vdash n$ is a partition, we write $V^\lambda$ for the corresponding irreducible module,
 $\chi^\lambda: \symm_n \rightarrow \FFF$ for its character,
and $f^\lambda := \dim V^\lambda$ for its dimension.
The number $f^\lambda$ counts standard tableaux of shape $\lambda$.

The vector space $\Class(\symm_n,\FFF)$ of $\FFF$-valued class functions on $\symm_n$ has basis 
$\{ \chi^\lambda \,:\, \lambda \vdash n \}$ given by irreducible characters.
The {\em Kronecker product} $*$ on $\Class(\symm_n,\FFF)$ is defined by
\begin{equation}
(\varphi * \psi)(w) := \varphi(w) \cdot \psi(w)
\end{equation}
for any $\varphi, \psi \in \Class(\symm_n,\FFF)$ and $w \in \symm_n$.
If $V_1$ and $V_2$ are $\symm_n$-modules, their vector space tensor product $V_1 \otimes V_2$ carries a diagonal action of $\symm_n$
by the rule $w \cdot (v_1 \otimes v_2) := (w \cdot v_1) \otimes (w \cdot v_2)$.
The characters $\chi_{V_1}, \chi_{V_2}, \chi_{V_1 \otimes V_2}: \symm_n \rightarrow \FFF$ of these modules are related by 
$\chi_{V_1 \otimes V_2} = \chi_{V_1} * \chi_{V_2}$.

\section{Hilbert series and standard monomial basis}
\label{Hilbert}

\subsection{The injection relations}
In order to analyze the quotients $\FFF[\xx_{n \times n}] / I_n$, we start by exhibiting strategic elements of the ideal $I_n$.
Given two subsets $S, T \subseteq [n]$, define elements $a_{S,T}, b_{S,T} \in \FFF[\xx_{n \times n}]$ by 
\begin{equation}
a_{S,T} := \sum_{f: S \hookrightarrow T} \left( \prod_{ i \in S } x_{i,f(i)} \right) \quad \quad \text{and} \quad \quad
b_{S,T} := \sum_{f: S \hookrightarrow T} \left( \prod_{ i \in S } x_{f(i),i}  \right)
\end{equation}
where both sums are over injective functions $f: S \hookrightarrow T$. 
For example, if $S = \{2,4\}$ and $T = \{1,3,4\}$ we have 
\begin{align*}
a_{S,T}  &= x_{2,1} x_{4,3} + x_{2,1} x_{4,4} + x_{2,3} x_{4,1} + x_{2,3} x_{4,4} + x_{2,4} x_{4,1} + x_{2,4} x_{4,3}, \\
b_{S,T}  &= x_{1,2} x_{3,4} + x_{1,2} x_{4,4} + x_{3,2} x_{1,4} + x_{3,2} x_{4,4} + x_{4,2} x_{1,4} + x_{4,2} x_{3,4}.
\end{align*}
In general, the polynomials $a_{S,T}$ and $b_{S,T}$ are obtained from one another by transposing the matrix $\xx_{n \times n}$ of variables.
We have $a_{S,T} = b_{S,T} = 0$ whenever $|S| > |T|$.

Since the product of any two variables in the same row or column of $\xx_{n \times n}$ is a generator of $I_n$, we have the congruences
\begin{equation}
\label{all-functions-congruences}
a_{S,T} \equiv \prod_{i \in S} \left( \sum_{j \in T} x_{i,j} \right) \mod I_n \quad \quad \text{and} \quad \quad
b_{S,T} \equiv \prod_{i \in S} \left( \sum_{j \in T} x_{j,i} \right)  \mod I_n
\end{equation}
modulo $I_n$.
In other words, as far as the quotient $\FFF[\xx_{n \times n}]/I_n$ is concerned, we could have defined $a_{S,T}$ and $b_{S,T}$ using all functions
 $S \rightarrow T$, not just injections.
Our first lemma states that $a_{S,T}$ and $b_{S,T}$ are members of $I_n$ provided that $|S| + |T| > n$.

\begin{lemma}
\label{injective-lemma}
Let $S, T \subseteq [n]$ be subsets.  If $|S| + |T| > n$ we have $a_{S,T}, b_{S,T} \in I_n$.
\end{lemma}


\begin{proof}
The polynomial $b_{S,T}$ is obtained from $a_{S,T}$ by transposing the matrix $\xx_{n \times n}$ of variables, an operation under which $I_n$ is stable.
As such, it suffices to prove the lemma for $a_{S,T}$. Furthermore, by the congruences \eqref{all-functions-congruences} it suffices to prove
the lemma when $|S| + |T| = n+1$.
Finally, since $I_n$ is stable under the action of the product group $\symm_n \times \symm_n$ on the rows and columns of $\xx_{n \times n}$, it is enough
 to consider the case where $S = [s]$ and $T = [t]$ for $s + t = n+1$.
 
 We argue by increasing induction on $s$ (and decreasing induction on $t$).  If $s = 1$ then $t = n$ and $a_{S,T} = x_{1,1} + x_{1,2} + \cdots + x_{1,n}$ is a generator
 of the ideal $I_n$. If $s > 1$, we have
 \begin{equation}
 a_{S,T} \equiv \prod_{i = 1}^s \left( \sum_{j = 1}^t x_{i,j} \right) = (x_{1,1} + x_{1,2} + \cdots + x_{1,t}) \times \left[ \prod_{i = 2}^s \left( \sum_{j = 1}^{t+1} x_{i,j} \right)  \right] - \EE
 \end{equation}
 where the congruence modulo $I_n$ follows from \eqref{all-functions-congruences},  the  expression $[ \, \cdots ]$ in square brackets lies in $I_n$ 
 by induction, and the ``error term" $\EE$ is given by
 \begin{equation}
 \EE =  (x_{1,1} + x_{1,2} + \cdots + x_{1,t}) \times \sum_{\varnothing \neq S' \subseteq \{2, \dots, s\}} \left( \prod_{i' \in S'} x_{i',t+1} \times  \prod_{i \in \{2,\dots,s\} - S'} (x_{i,1} + \cdots + x_{i,t})  \right).
 \end{equation}
 It suffices to show that $\EE \in I_n$.  When the $|S'| > 1$ and $i'_1, i'_2 \in S'$ are distinct, the corresponding summand in $\EE$ contains the product 
 $x_{i'_1,t+1} \cdot x_{i'_2,t+1}$, and so lies in $I_n$. We conclude that 
 \begin{equation}
 \EE \equiv 
 (x_{1,1} + x_{1,2} + \cdots + x_{1,t}) \times \sum_{i_0 =  2}^s  \left( x_{i_0,t+1} \times  \prod_{2 \leq i \leq s}^{i \neq i_0} (x_{i,1} + \cdots + x_{i,t})  \right)
 \end{equation}
 modulo $I_n$. Applying the congruences \eqref{all-functions-congruences} and the defining relations of $I_n$, we arrive at
   \begin{equation}
   \label{e-final}
 \EE \equiv  \pm
 (x_{1,t+1} + x_{1,t+2} + \cdots + x_{1,n}) \times \sum_{i_0 =  2}^s  \left( x_{i_0,t+1} \times  \prod_{2 \leq i \leq s}^{i \neq i_0} (x_{i,t+2} + x_{i,t+3} \cdots + x_{i,n})  \right)
 \end{equation}
 The sum $(x_{i,t+2} + x_{i,t+3} \cdots + x_{i,n}) $ contains $n-t-1 = n - (n+1-s) - 1 = s - 2$ terms. The Pigeonhole Principle
 implies that every term in the expansion of the RHS of 
 the congruence \eqref{e-final} will contain a product of variables $x_{i,j} \cdot x_{i',j}$ for some $i \neq i'$ so that $\EE \in I_n$. 
 We conclude that $a_{S,T} \in I_n$ and the lemma is proven.
\end{proof}

\subsection{Shadow sets}

We represent a permutation $w = [w(1), \dots, w(n)] \in \symm_n$ with its {\em graph}, i.e. the collection of points $\{ (i,w(i)) \,:\, 1 \leq i \leq n \}$ on the grid $[n] \times [n]$.
For example, the permutation $w = [4,1,8,5,3,6,2,7] \in \symm_8$ is given below in bullets. 

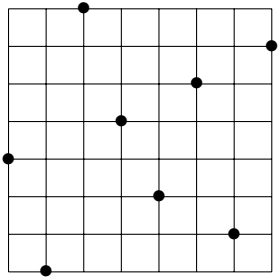
\begin{center}
\begin{tikzpicture}[scale = 0.5]

\draw (1,1) grid (8,8);

\node at (1,4) {$\bullet$};
\node at (2,1) {$\bullet$};
\node at (3,8) {$\bullet$};
\node at (4,5) {$\bullet$};
\node at (5,3) {$\bullet$};
\node at (6,6) {$\bullet$};
\node at (7,2) {$\bullet$};
\node at (8,7) {$\bullet$};

\end{tikzpicture}
\end{center}

Viennot used \cite{Viennot} the graph of a permutation $w$ to obtain its image $(P(w),Q(w))$ under the Schensted correspondence as follows.
Shine a flashlight northeast from the origin (0,0). Each bullet in the permutation casts a shadow to its northeast. The boundary of the shaded region is the 
{\em first shadow line}; in our example it is as follows.
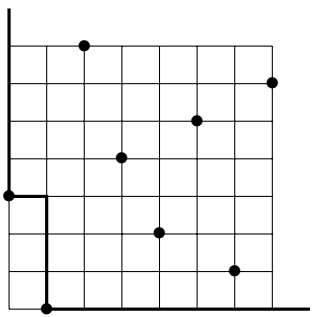
\begin{center}
\begin{tikzpicture}[scale = 0.5]

\draw (1,1) grid (8,8);

\node at (1,4) {$\bullet$};
\node at (2,1) {$\bullet$};
\node at (3,8) {$\bullet$};
\node at (4,5) {$\bullet$};
\node at (5,3) {$\bullet$};
\node at (6,6) {$\bullet$};
\node at (7,2) {$\bullet$};
\node at (8,7) {$\bullet$};

\draw[very thick] (1,9) -- (1,4) -- (2,4) -- (2,1) -- (9,1);

\end{tikzpicture}
\end{center}
Removing the points on the first shadow line and repeating this procedure, we obtain the {\em second shadow line}. 
Iterating, we obtain the {\em third shadow line}, the {\em fourth shadow line}, and so on.
In our example, the shadow lines are shown below.
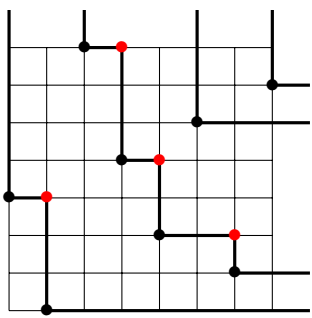
\begin{center}
\begin{tikzpicture}[scale = 0.5]

\draw (1,1) grid (8,8);

\node at (1,4) {$\bullet$};
\node at (2,1) {$\bullet$};
\node at (3,8) {$\bullet$};
\node at (4,5) {$\bullet$};
\node at (5,3) {$\bullet$};
\node at (6,6) {$\bullet$};
\node at (7,2) {$\bullet$};
\node at (8,7) {$\bullet$};

\draw[very thick] (1,9) -- (1,4) -- (2,4) -- (2,1) -- (9,1);

\draw[very thick] (3,9) -- (3,8) -- (4,8) -- (4,5) -- (5,5) -- (5,3) -- (7,3) -- (7,2) -- (9,2);

\draw[very thick] (6,9) -- (6,6) -- (9,6);

\draw[very thick] (8,9) -- (8,7) -- (9,7);

\node at (2,4) {${\color{red} \bullet}$};
\node at (4,8) {${\color{red} \bullet}$};
\node at (5,5) {${\color{red} \bullet}$};
\node at (7,3) {${\color{red} \bullet}$};

\end{tikzpicture}
\end{center}

Let $w \in \symm_n$ and suppose that the shadow lines of $w$ are given by $L_1, \dots, L_r$ from southwest to northeast.
Viennot proved \cite{Viennot} that if $w \mapsto (P(w), Q(w))$ under the Schensted correspondence then
the $y$-coordinates of the infinite horizontal rays in $L_1, \dots, L_r$ form the first row of $P(w)$ and the $x$-coordinates of the infinite vertical rays 
of $L_1, \dots, L_r$ form the first row of $Q(w)$.  
In our example, the first row of $P(w)$ is $\begin{young} 1 & 2 & 6 & 7 \end{young}$ while the first row of $Q(w)$ is $\begin{young} 1 & 3 & 6 & 8 \end{young}$.
In particular, the common length of the first row of $P(w)$ and $Q(w)$ is the number of shadow lines.
The northeast corners of the shadow lines played an important role in Viennot's work, and will for us, as well.

\begin{defn}
The {\em shadow set} $\SSS(w)$ of a permutation $w \in \symm_n$ is the collection of points $(i,j)$ in the grid $[n] \times [n]$ which lie at the northeast corner of 
a shadow line of $w$.
\end{defn}

In our example, the points in the shadow set 
$\SSS(w) = \{ (2,4), (4,8), (5,5), (7,3) \}$ are drawn in red.
For any permutation $w \in \symm_n$, the shadow set $\SSS(w)$ contains at most one point in any row or column. Such subsets of the square grid have a name.

\begin{defn}
A subset $\RRR \subseteq [n] \times [n]$ is a {\em (non-attacking) rook placement} if $\RRR$ contains at most one point in any row or column.
\end{defn}

Rook placements are also known as `partial permutations'.
Importantly, the Viennot shadow line construction may be performed on an arbitrary rook placement, not just on the graph of a permutation.

Although every permutation shadow set is a rook placement, not every rook placement is the shadow set of a permutation.
 For example, shadow sets contain no points in row 1 or column 1.
In Lemma~\ref{ballot-criterion-lemma} below, we give a combinatorial criterion for deciding whether a rook placement is a shadow set.

Returning to our permutation $w \in \symm_n$,
we may iterate the shadow line construction on the shadow set $\SSS(w)$. In our $n = 8$ example this yields the shadow lines.
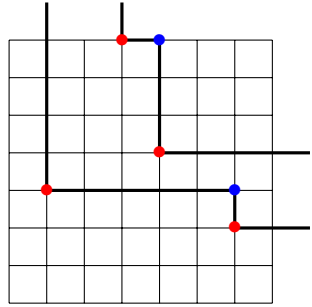
\begin{center}
\begin{tikzpicture}[scale = 0.5]

\draw (1,1) grid (8,8);

\draw [very thick] (2,9) -- (2,4) -- (7,4) -- (7,3) -- (9,3);

\draw [very thick] (4,9) -- (4,8) -- (5,8) -- (5,5) -- (9,5);

\node at (2,4) {${\color{red} \bullet}$};
\node at (4,8) {${\color{red} \bullet}$};
\node at (5,5) {${\color{red} \bullet}$};
\node at (7,3) {${\color{red} \bullet}$};

\node at (5,8) {${\color{blue} \bullet}$};
\node at (7,4) {${\color{blue} \bullet}$};

\end{tikzpicture}
\end{center}
Viennot proved that the horizontal and vertical rays of these `iterated' shadow lines give the second rows of $P(w)$ and $Q(w)$, respectively.
In our example, the second row of $P(w)$ is $\begin{young} 3 & 5 \end{young}$ and the second row of $Q(w)$ is $\begin{young} 2 & 4 \end{young}$.
These iterated shadow lines produce an iterated shadow set $\SSS(\SSS(w))$ whose points are drawn in blue.
Repeating this procedure in our example yields the iterated shadow sets and shadow lines
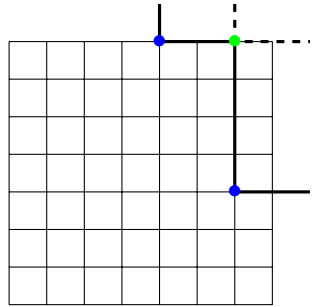
\begin{center}
\begin{tikzpicture}[scale = 0.5]

\draw (1,1) grid (8,8);

\draw [very thick] (5,9) -- (5,8) -- (7,8) -- (7,4) -- (9,4);

\draw [very thick, dashed] (7,9) -- (7,8) -- (9,8);

\node at (5,8) {${\color{blue} \bullet}$};
\node at (7,4) {${\color{blue} \bullet}$};

\node at (7,8) {${\color{green} \bullet}$};

\end{tikzpicture}
\end{center}
and we conclude that the tableaux $P(w)$ and $Q(w)$ are given by
\begin{equation*}
\begin{young}
1 & 2 & 6 & 7 \cr
3 & 5 \cr
4 \cr 
8
\end{young} \quad \quad \text{and} \quad \quad
\begin{young}
1 & 3 & 6 & 8 \cr
2 & 4 \cr
5 \cr
7
\end{young},
\end{equation*}
respectively.

\begin{theorem} (Viennot \cite{Viennot})
\label{viennot-theorem}
The shadow line procedure described above computes the image $(P(w), Q(w))$ of a permutation $w \in \symm_n$ under the Schensted correspondence.
\end{theorem}

For our purposes,
we may take Theorem~\ref{viennot-theorem} as the definition of the Schensted correspondence.
Combining Theorem~\ref{viennot-theorem} with Schensted's Theorem~\ref{schensted-theorem} yields the following result immediately.

\begin{lemma}
\label{shadow-size-lemma}
Let $w \in \symm_n$.  The size $|\SSS(w)|$ of the shadow set of $w$ is given by
\begin{equation}
|\SSS(w)| = n - \lis(w).
\end{equation}
\end{lemma}

We close this subsection with a combinatorial criterion for deciding when a rook placement $\RRR$ is the shadow set of some permutation $w \in \symm_n$.
We use the fact that the shadow line construction may be applied to 
$\RRR$.  This will yield a pair $(P,Q)$ of partial standard tableaux with the same shape such that the 
$y$-coordinates of $\RRR$ are the entries in $P$ and the $x$-coordinates in $\RRR$ are the entries in $Q$.

\begin{lemma}
\label{ballot-criterion-lemma}
Let $\RRR \subseteq [n] \times [n]$ be a rook placement and apply the shadow line construction to $\RRR$.
Let $L_1, \dots, L_r$ be the shadow lines so obtained.
Define two length $n$ sequences $x_1 x_2 \dots x_n$ and $y_1 y_2 \dots y_n$ over the alphabet $\{1,0,-1\}$ by
\begin{equation}
x_i = \begin{cases}
1 & \text{if one of the lines $L_1, \dots, L_r$ has a vertical ray at $x = i$,} \\
-1 & \text{if the vertical line $x = i$ does not meet $\RRR$,} \\
0 & \text{otherwise.}
\end{cases}
\end{equation}
and
\begin{equation}
y_i = \begin{cases}
1 & \text{if one of the lines $L_1, \dots, L_r$ has a horizontal ray at $y = i$,} \\
-1 & \text{if the horizontal line $y = i$ does not meet $\RRR$,} \\
0 & \text{otherwise.}
\end{cases}
\end{equation}
Then $\RRR = \SSS(w)$ is the shadow set of some permutation $w \in \symm_n$ if and only if for all $1 \leq i \leq n$ we have 
$x_1 + x_2 + \cdots + x_i \leq 0$ and $y_1 + y_2 + \cdots + y_i \leq 0$.
\end{lemma}

\begin{proof}
Suppose $\RRR = \SSS(w)$ is the shadow set of a permutation $w \in \symm_n$. If $w \mapsto (P(w), Q(w))$ under the Schensted correspondence,
the horizontal rays of $L_1, \dots, L_r$ give the second row of $P(w)$ and the vertical rays of $L_1, \dots, L_r$ give the second row of $Q(w)$.
The $y$-coordinates which do not appear in $\RRR$ give the first row of $P(w)$ and the $x$-coordinates which do not appear in $\RRR$ give the first row of $Q(w)$.
Since $P(w)$ and $Q(w)$ are standard all prefix sums of the sequences $x_1 x_2 \dots x_n$ and $y_1 y_2 \dots y_n$ are nonpositive.

Now assume that all prefix sums of $x_1 x_2 \dots x_n$ and $y_1 y_2 \dots y_n$ are nonpositive. We may apply Viennot's construction to the set $\RRR$ to get a 
pair $(P',Q')$ of partial standard tableaux where the entries of $P'$ are the $y$-coordinates in $\RRR$
and the entries of $Q'$ are the $x$-coordinates in $\RRR$.
By the assumption on prefixes, the tableaux $P$ and $Q$ obtained by adding a first row to $P$ and $Q$ consisting of those $y$-coordinates and $x$-coordinates which
do not appear in $\RRR$ (respectively) are both standard.  If we let $w \in \symm_n$ be the unique permutation such that $w \mapsto (P,Q)$, 
Viennot's Theorem~\ref{viennot-theorem} implies that $\SSS(w) = \RRR$.
\end{proof}

An example may  help in understanding Lemma~\ref{ballot-criterion-lemma} and its proof. Let $n = 8$ and let $\RRR$ be the rook placement
\begin{equation*}
\SSS = \{ (2,8), (3,7), (5,3), (6,5), (7,6) \}
\end{equation*}
of size 5.
Applying the Viennot shadow line construction to $\RRR$ yields
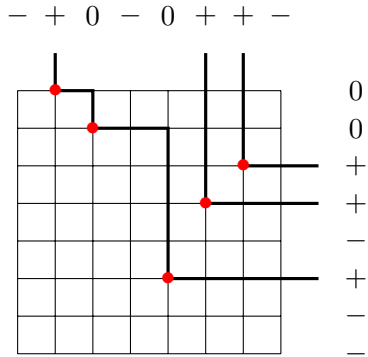
\begin{center}
\begin{tikzpicture}[scale = 0.5]

\draw (1,1) grid (8,8);

\draw[very thick] (2,9) -- (2,8) -- (3,8) -- (3,7) -- (5,7) -- (5,3) -- (9,3);

\draw[very thick] (6,9) -- (6,5) -- (9,5);

\draw[very thick] (7,9) -- (7,6) -- (9,6);

\node at (2,8) {${\color{red} \bullet}$};
\node at (3,7) {${\color{red} \bullet}$};
\node at (5,3) {${\color{red} \bullet}$};
\node at (6,5) {${\color{red} \bullet}$};
\node at (7,6) {${\color{red} \bullet}$};

\node at (1,10) {$-$};
\node at (2,10) {$+$};
\node at (3,10) {$0$};
\node at (4,10) {$-$};
\node at (5,10) {$0$};
\node at (6,10) {$+$};
\node at (7,10) {$+$};
\node at (8,10) {$-$};

\node at (10,1) {$-$};
\node at (10,2) {$-$};
\node at (10,3) {$+$};
\node at (10,4) {$-$};
\node at (10,5) {$+$};
\node at (10,6) {$+$};
\node at (10,7) {$0$};
\node at (10,8) {$0$};
\end{tikzpicture}
\end{center}
where the sequences $x_1 x_2 \dots x_8$ and $y_1 y_2 \cdots y_8$ in $\{1,0,-1\}$ are shown horizontally and vertically, respectively.
A $+1$ in a given row (or column) corresponds to an infinite ray of a shadow line, a 0 corresponds to a shadow line segment which is not an infinite ray, and a $-1$ corresponds
to that row (or column) not containing an element of $\RRR$.
We have $x_1 + x_2 + \cdots + x_7 = 1 > 0$, so by Lemma~\ref{ballot-criterion-lemma} the set $\RRR$ is not the shadow set of a permutation in $\symm_8$.
Indeed, applying Schensted insertion to the rook placement $\RRR$ yields the pair of tableaux
$P'$ and $Q'$ given by
\begin{equation*}
\begin{young}
 3 & 5 & 6 \cr 7 \cr 8
\end{young} \quad \quad \text{and} \quad \quad
\begin{young}
2 & 6 & 7 \cr 3 \cr 5
\end{young}
\end{equation*}
(respectively)
and adding the row $\begin{young} 1 & 4 & 8 \end{young}$ corresponding to the positions of the $-1$'s in the sequence $x_1 x_2 \dots x_8$
to the top row of $Q'$ would not yield a standard tableau.

\subsection{Shadow monomials and spanning}
Our next task is to convert the combinatorics of the previous subsection into a spanning set for the quotient ring $\FFF[\xx_{n \times n}] / I_n$.
Given any  set
$\SSS \subseteq [n] \times [n]$ of grid points,
let $m(\SSS) = \prod_{(i,j) \in \SSS} x_{i,j}$ be the corresponding squarefree monomial in $\FFF[\xx_{n \times n}]$.

\begin{lemma}
\label{squarefree-span-lemma}
The family of monomials $m(\RRR)$ corresponding to rook placements $\RRR \subseteq [n] \times [n]$ descends to a spanning
set of $\FFF[\xx_{n \times n}]/I_n$.
\end{lemma}

\begin{proof}
This is immediate from the fact that generating set of
 $I_n$ contains all squares $x_{i,j}^2$ of variables and all products of two variables in a given row or column.
\end{proof}

The spanning set of Lemma~\ref{squarefree-span-lemma} is far from a basis.
In order to extract a basis from this spanning set, we introduce a strategic term order.
Recall that the {\em lexicographical order} on monomials in an ordered set of variables $y_1 > y_2 > \cdots > y_N$ is given by
$y_1^{a_1} \cdots y_N^{a_N} < y_1^{b_1} \cdots y_N^{b_N}$ if there exists $1 \leq j \leq N$ with $a_i = b_i$ for $i < j$ and $a_j < b_j$.

\begin{defn}
\label{toeplitz-defn}
The {\em Toeplitz term order} $<_\Top$ on monomials in $\FFF[\xx_{n \times n}]$ is the lexicographical term order with respect to the order on variables given by
\begin{equation}
x_{1,1} > x_{2,1} > x_{1,2} > x_{3,1} > x_{2,2} > x_{1,3} > \cdots > x_{n,n-1} > x_{n-1,n} > x_{n,n}.
\end{equation}
\end{defn}

Roughly speaking, the Toeplitz term order weights a variable $x_{a,b}$ heavier than $x_{c,d}$ whenever $a + b < c + d$, and then breaks ties lexicographically. 
In fact, this tie breaking process among variables $x_{i,j}$ with $i + j$ constant will be irrelevant for the arguments that follow; all that is important is the relative weight of 
the variables $x_{i,j}$ for which $i + j$ differs.
The word ``Toeplitz" comes from Toeplitz matrices (which are constant along diagonals).
Since all of the relations we apply will be homogeneous, we could have also defined $<_\Top$ by ordering by total degree first, and then using the lexicographical order with
respect to
the indicated variable order to break ties.

\begin{defn}
\label{shadow-defn}
Let $w \in \symm_n$.  The {\em shadow monomial} $\sss(w) \in \FFF[\xx_{n \times n}] / I_n$ is the squarefree monomial corresponding to the shadow set of $w$.
In symbols, we have $\sss(w) = m(\SSS(w))$.
\end{defn}

For example, if $w = [4,1,8,5,3,6,2,7] \in \symm_8$ we have $\SSS(w) = \{ (2,4), (4,8), (5,5), (7,3) \}$ so that 
$\sss(w) = x_{2,4} \cdot x_{4,8} \cdot x_{5,5} \cdot x_{7,3}$.
Our next lemma shows  that the shadow monomials of permutations span the quotient ring $\FFF[\xx_{n \times n}] / I_n$.  The key tools in its proof
are the relations in $\FFF[\xx_{n \times n}]/I_n$ coming from
Lemma~\ref{injective-lemma} and the characterization 
(Lemma~\ref{ballot-criterion-lemma}) of when a rook placement monomial $m(\RRR)$ is the shadow monomial
$\sss(w)$
of a permutation $w \in \symm_n$.
To begin, we record the $<_\Top$-leading terms of the elements of $I_n$ appearing in Lemma~\ref{injective-lemma}.

\begin{observation}
\label{leading-observation}
Let $S = \{s_1 < \cdots < s_p \}$ and $T = \{t_1 < \cdots < t_q \}$ be subsets of $[n]$ with $p \leq q$. Then
\begin{equation}
\initial_{<_\Top}(a_{S,T}) = x_{s_1,t_1} x_{s_2,t_2} \cdots x_{s_p,t_p} \quad \quad \text{and} \quad \quad
\initial_{<_\Top}(b_{S,T}) = x_{t_1,s_1} x_{t_2,s_2} \cdots x_{t_p,s_p}.
\end{equation}
\end{observation}

In other words, the leading monomials 
of $a_{S,T}$ and $b_{S,T}$ correspond to the injection $S \hookrightarrow T$ which assigns the elements of $S$ to the smallest $|S|$ elements of $T$
in an order-preserving fashion. If $S = \{2,4\}$ and $T = \{1,4,5\}$ then $a_{S,T}$ given by
\begin{equation*}
a_{S,T} =  \underline{x_{2,1} x_{4,4}} + x_{2,4} x_{4,1} + x_{2,1} x_{4,5} + x_{2,5} x_{4,1} + x_{2,4} x_{4,5} + x_{2,5} x_{4,4}
\end{equation*}
with its $<_\Top$-leading term underlined.
We have all the pieces we need to prove our spanning result.

\begin{lemma}
\label{spanning-lemma}
The shadow monomials $\{ \sss(w) \,:\, w \in \symm_n \}$ descend to a spanning set of the quotient ring $\FFF[\xx_{n \times n}] / I_n$.
\end{lemma}

\begin{proof}
Let $\RRR \subseteq [n] \times [n]$ be a rook placement.
By Lemma~\ref{squarefree-span-lemma} it suffices to show that $m(\RRR)$ lies in the span of 
$\{ \sss(w) \,:\, w \in \symm_n \}$ modulo $I_n$.
If $\RRR = \SSS(w)$ for some permutation $w \in \symm_n$ then $m(\RRR) = \sss(w)$ and this is clear, so assume that 
$\RRR \neq \SSS(w)$ for all $w \in \symm_n$.

Apply Viennot's shadow line construction to the rook placement $\RRR$.
Let $L_1, \dots, L_r$ be the shadow lines so obtained, ordered from southwest to northeast, and let $x_1 x_2 \dots x_n$ and $y_1 y_2 \dots y_n$ 
be the sequences appearing in the statement of
Lemma~\ref{ballot-criterion-lemma}.
Since $\RRR$ is not the shadow set of a permutation, Lemma~\ref{ballot-criterion-lemma} implies that at least one of the sequences 
$x_1 x_2 \dots x_n$ and $y_1 y_2 \dots y_n$  has a prefix with a strictly positive sum.
We assume that $x_1 x_2 \dots x_n$ has a prefix with strictly positive sum; the case of $y_1 y_2 \dots y_n$ is similar. 

Choose $1 \leq a \leq n$ minimal such that 
$x_1 + x_2 + \cdots + x_a > 0$. By the minimality of $a$, we have $x_a = 1$ so that $x = a$ is the vertical ray of one of the shadow lines $L_p$ for some 
$1 \leq p \leq r$.  We define a size $p$ subset $\{ (i_1, j_1), \dots, (i_p, j_p) \} \subseteq \RRR$ as follows. 
Starting at the vertical ray of $L_p$, let $(i_p,j_p)$ be the first element of $\RRR$
encountered by marching south (in particular, we have $i_p = a$).  Now march west from $(i_p,j_p)$ until one encounters a vertical segment of the shadow line $L_{p-1}$. March south
along this segment until one reaches a point $(i_{p-1},j_{p-1}) \in \RRR$.  Now march west from $(i_{p-1},j_{p-1})$ until one encounters a vertical segment
of the shadow line $L_{p-2}$. March south along this segment until one reaches a point $(i_{p-2}, j_{p-2}) \in \RRR$. Continuing this process, we arrive at a subset 
$\{ (i_1, j_1), \dots, (i_p, j_p) \} \subseteq \RRR$ such that
\begin{itemize}
\item the point $(i_q, j_q)$ lies on the shadow line $L_q$ for each $1 \leq q \leq p$,
\item we have $i_1 < \cdots < i_p$, and 
\item we have $j_1 < \cdots < j_p$.
\end{itemize}
Let $\RRR' := \RRR - \{ (i_1, j_1), \dots, (i_p, j_p) \}$ be the complement of $\{ (i_1, j_1), \dots, (i_p, j_p) \}$ in $\RRR$.

An example may help in understanding these constructions.  
Let $n = 11$ and consider the rook placement
$$ 
\RRR = \{ (2,9), (3,8), (4,3), (6,2), (7,6), (8,7), (9,5), (11,11) \}. 
$$
By Lemma~\ref{ballot-criterion-lemma}, the rook placement $\RRR$ is not the shadow set of a permutation
in $\symm_8$ because $x_1 + x_2 + \cdots + x_8 = 1 > 0$. 
Furthermore, this is the smallest positive sum prefix of the word $x_1 x_2 \dots x_{11}$.
We conclude that $a = 8$.
Our marching procedure on the shadow line diagram of $\RRR$ is shown in dashed and blue as follows.
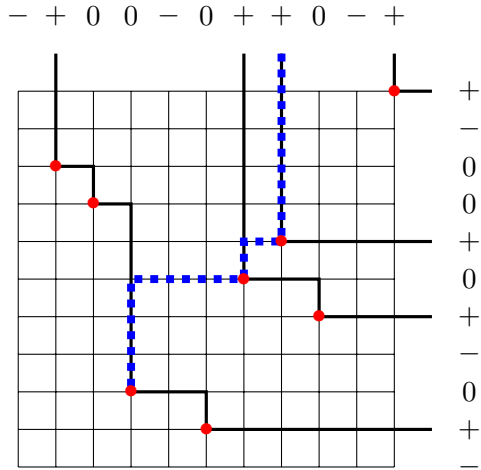
\begin{center}
\begin{tikzpicture}[scale = 0.5]

\draw (0,0) grid (10,10);

\draw[very thick] (1,11) -- (1,8) -- (2,8) -- (2,7) -- (3,7) -- (3,2) -- (5,2) -- (5,1) -- (11,1);

\draw[very thick] (6,11) -- (6,5) -- (8,5) -- (8,4) -- (11,4);

\draw[very thick] (7,11) -- (7,6) -- (11,6);

\draw[very thick] (10,11) -- (10,10) -- (11,10);

\draw[line width = 3pt, dashed, color = blue] (7,11) -- (7,6) -- (6,6) -- (6,5) -- (3,5) -- (3,2);

\node at (1,8) {${\color{red} \bullet}$};
\node at (2,7) {${\color{red} \bullet}$};
\node at (3,2) {${\color{red} \bullet}$};
\node at (6,5) {${\color{red} \bullet}$};
\node at (7,6) {${\color{red} \bullet}$};
\node at (8,4) {${\color{red} \bullet}$};
\node at (5,1) {${\color{red} \bullet}$};
\node at (10,10) {${\color{red} \bullet}$};

\node at (0,12) {$-$};
\node at (1,12) {$+$};
\node at (2,12) {$0$};
\node at (3,12) {$0$};
\node at (4,12) {$-$};
\node at (5,12) {$0$};
\node at (6,12) {$+$};
\node at (7,12) {$+$};
\node at (8,12) {$0$};
\node at (9,12) {$-$};
\node at (10,12) {$+$};

\node at (12,0) {$-$};
\node at (12,1) {$+$};
\node at (12,2) {$0$};
\node at (12,3) {$-$};
\node at (12,4) {$+$};
\node at (12,5) {$0$};
\node at (12,6) {$+$};
\node at (12,7) {$0$};
\node at (12,8) {$0$};
\node at (12,9) {$-$};
\node at (12,10) {$+$};
\end{tikzpicture}
\end{center}
We conclude that $(i_1, j_1) = (4,3), (i_2, j_2) = (7,6),$ and $(i_3, j_3) = (8,7)$.  Furthermore, we have the set
\begin{equation*}
\RRR' = \RRR - \{ (i_1,j_1), (i_2,j_2), (i_3,j_3) \} = \{ (2,9), (3,8), (6,2),  (9,5), (11,11) \}
\end{equation*}
of rooks in $\RRR$ which are not visited by the dashed blue line.

Consider the squarefree monomial $m(\RRR')$ corresponding to the rooks in $\RRR' \subseteq \RRR$ which are not reached by our marching procedure.
The ideal
$m(\RRR') \cdot \FFF[\xx_{n \times n}]/I_n$  generated by $m(\RRR')$ in the ring $\FFF[\xx_{n \times n}]/I_n$ admits a morphism from a smaller quotient of the same form.
More precisely, let  $\bar{n} := n - |\RRR'|$ and let $\bar{\xx}$ be the $\bar{n} \times \bar{n}$ matrix of variables
\begin{equation}
\bar{\xx} = \{ x_{i,j} \,:\, \text{neither the vertical line $x = i$ nor the horizontal line $y = j$ meet the set $\RRR'$} \}.
\end{equation}
In our example above, the matrix $\bar{\xx}$ consists of the variables $x_{i,j}$ indexed by $i \in \{1,4,5,7,8,10\}$ and $j \in \{1,3,4,6,7,10\}$.
Let $\FFF[\bar{\xx}]$ be the polynomial ring over the variables in $\bar{\xx}$ and let $\bar{I} \subseteq \FFF[\bar{\xx}]$ be the natural copy of
the ideal $I_{\bar{n}}$ in the square variable matrix $\bar{\xx}$.
The map
\begin{equation}
\label{redundant-isomorphism}
\varphi: \FFF[\bar{\xx}] / \bar{I}  \longrightarrow m(\RRR') \cdot \FFF[\xx_{n \times n}]/I_n
\end{equation}
induced by $f \mapsto m(\RRR') \cdot f$ is easily seen to be a (well-defined) homomorphism of $\FFF[\bar{\xx}]$-modules; one simply checks that for any generator
$g \in \FFF[\bar{\xx}]$ of $\bar{I}$, we have $m(\RRR') \cdot g \in I_n$.
We consider the sets
\begin{equation}
T := \{ i_1 < i_2 < \cdots < i_p < i_p + 1 < i_p + 2 < \cdots < n \} - \{ i \,:\, (i,j) \in \RRR' \text{ for some $j$} \}
\end{equation}
and
\begin{equation}
S := \{ j_1 < j_2 < \cdots < j_p \}.
\end{equation}
In our example we have $T = \{4,7,8,10\}$ and $S = \{3,6,7\}$.

By the definitions of $S$ and $T$,
 the polynomial $b_{S,T} \in \FFF[\bar{\xx}]$ does not involve any of the 
 variables which share a row or column with a rook $(i,j) \in \RRR'$ which is not visited 
by our marching procedure.
Since $i_p = a$ and we have the prefix inequality $x_1 + x_2 + \cdots + x_a > 0$, we have $|S| + |T| > \bar{n}$. 
Lemma~\ref{injective-lemma} applies to give 
\begin{equation}
b_{S,T} \in \bar{I}.
\end{equation}
Since the map $\varphi$ of \eqref{redundant-isomorphism} is a homomorphism of $\FFF[\bar{\xx}]$-modules we obtain
\begin{equation}
\label{clincher-containment}
\varphi(b_{S,T}) = 
m(\RRR') \cdot b_{S,T} \in I_n.
\end{equation}
Observation~\ref{leading-observation} implies that the Toeplitz-leading term of $m(\RRR') \cdot b_{S,T}$ is $m(\RRR)$, so the
 membership \eqref{clincher-containment} yields
\begin{equation}
m(\RRR) \equiv \Sigma \mod I_n
\end{equation}
where $\Sigma$ is a $\FFF$-linear combination of monomials which are $<_\Top m(\RRR)$. By induction on the Toeplitz order, 
the lemma is proven.
\end{proof}

Lemma~\ref{spanning-lemma} (and its proof) give a Gr\"obner basis for the ideal $I_n \subseteq \FFF[\xx_{n \times n}]$ with respect to the Toeplitz order
which consists of
\begin{itemize}
\item any product of two variables in $\xx_{n \times n}$ which lie in the same row or column, and
\item in the notation of the proof of Lemma~\ref{spanning-lemma}, and polynomial of the form $m(\RRR') \cdot b_{S,T}$ for a rook placement
$\RRR \subseteq [n] \times [n]$ which is not the shadow set of a permutation $w \in \symm_n$ for which some prefix of the word $x_1 x_2 \dots x_n$ is positive,
and the image of $m(\RRR') \cdot b_{S,T}$ under the involution $\FFF[\xx_{n \times n}]$ which interchanges $x_{i,j}$ and $x_{j,i}$.
\end{itemize}
This Gr\"obner basis is far from minimal.
We leave the computation of a minimal (or reduced) Gr\"obner basis of $I_n$ as an open problem.

\subsection{Standard monomial basis and Hilbert series}
Lemma~\ref{spanning-lemma} bounds the quotient ring $\FFF[\xx_{n \times n}]/I_n$ from above by giving an $\FFF$-linear spanning set. In this subsection we use orbit harmonics to bound this quotient 
from below.

Let $\FFF^{n \times n}$ be the affine space of $n \times n$ matrices over $\FFF$ with coordinate ring $\FFF[\xx_{n \times n}]$.
Write $P_n \subseteq \FFF^{n \times n}$ for the locus of permutation matrices. That is, the set $P_n$ consists of
 0,1-matrices with a unique 1 in each row and column.
The vanishing ideal $\II(P_n) \subseteq \FFF[\xx_{n \times n}]$ of the permutation matrix locus is generated by
\begin{itemize}
\item $x_{i,j}^2 - x_{i,j}$ for all $1 \leq i, j \leq n$,
\item $x_{i,j} \cdot x_{i',j}$ for all $1 \leq i < i' \leq n$ and $i \leq j \leq n$,
\item $x_{i,j} \cdot x_{i,j'}$ for all $1 \leq i \leq n$ and $1 \leq j < j' \leq n$,
\item $x_{i,1} + \cdots + x_{i,n} - 1$ for all $1 \leq i \leq n$, and
\item $x_{1,j} + \cdots + x_{n,j} - 1$ for all $1 \leq j \leq n$.
\end{itemize}
Indeed, the generators in the first bullet point come from the $(i,j)$-entry of a permutation matrix being 0 or 1, the generators in the second and third bullet points come from products of distinct entries in a row or column of a permutation matrix vanishing, and the generators in the fourth and fifth bullet points come from the row and columns summing to 1.
Comparing these generators with Definition~\ref{i-defn}, we get the containment
\begin{equation}
\label{ideal-containment}
I_n \subseteq \gr \, \II(P_n).
\end{equation}
Although the highest degree components $\tau(g_1), \dots, \tau(g_r)$ of a generating set $\{ g_1, \dots, g_r \}$ of an ideal $I$
are in general insufficient to generate $\gr \, I$, in our case
 the containment \eqref{ideal-containment} is an equality.

\begin{theorem}
\label{vector-space-structure}
We have the equality of ideals $I_n = \gr \, \II(P_n)$ of $\FFF[\xx_{n \times n}]$. Furthermore, the set $\{ \sss(w) \,:\, w \in \symm_n \}$
of shadow monomials of permutations in $\symm_n$ descends to a basis of $\FFF[\xx_{n \times n}]/I_n$.
This is the standard monomial basis of $\FFF[\xx_{n \times n}]/I_n$ with respect to the Toeplitz term order $<_\Top$.
\end{theorem}

Standard monomial bases of quotient rings $\FFF[\xx]/I$ can be unpredictable, even for nicely presented ideals $I$.
However,
Theorem~\ref{vector-space-structure} informally suggests that the Toeplitz term order $<_\Top$ and the homogeneous ideal $I_n$ 
``know" the Viennot shadow line incarnation of the Schensted correspondence $w \mapsto (P(w), Q(w))$.

\begin{proof}
The chain \eqref{long-chain} of $\FFF$-vector space isomorphisms coming from orbit harmonics  reads
\begin{equation}
\FFF[P_n] \cong \FFF[\xx_{n \times n}] / \II(P_n) \cong \FFF[ \xx_{n \times n} ] / \gr \, \II(P_n).
\end{equation}
Lemma~\ref{spanning-lemma} and the containment \eqref{ideal-containment} of ideals yield the chain of (in)equalities
\begin{equation}
n! = | P_n | = \dim \FFF[ \xx_{n \times n} ] / \gr \, \II(P_n) \leq \dim \FFF[\xx_{n \times n}] / I_n \leq n!
\end{equation}
which forces $I_n = \gr \, \II(P_n)$ and $\dim \FFF[\xx_{n \times n}] / I_n = n!$.  Another application of Lemma~\ref{spanning-lemma} shows that 
the spanning set $\{ \sss(w) \,:\, w \in \symm_n \}$ of $\FFF[\xx_{n \times n}]/I_n$ is in fact a basis.
The proof of Lemma~\ref{spanning-lemma} shows that $\{ \sss(w) \,:\, w \in \symm_n \}$  is the standard monomial basis of 
$\FFF[\xx_{n \times n}]/I_n$ with respect to $<_\Top$.
\end{proof}

As a corollary, we get our promised relationship between the Hilbert series of $\FFF[\xx_{n \times n}]/I_n$ and longest increasing subsequences in permutations.

\begin{corollary}
\label{hilbert-series}
Let $a_{n,k}$ be the number of permutations in $\symm_n$ whose longest increasing sequence has length $k$. The quotient ring 
$\FFF[\xx_{n \times n}]/I_n$ has Hilbert series
\begin{equation}
\Hilb( \FFF[\xx_{n \times n}]/I_n; q) = a_{n,n} + a_{n,n-1} \cdot q + \cdots + a_{n,1} \cdot q^{n-1}.
\end{equation}
\end{corollary} 

\begin{proof}
Combine Lemma~\ref{shadow-size-lemma} and Theorem~\ref{vector-space-structure}.
\end{proof}

\subsection{Local permutation statistics}
Corollary~\ref{hilbert-series} gives the structure of $\FFF[\xx_{n \times n}]/I_n$ as a graded vector space. Our next goal is the structure of this quotient as a graded $\symm_n \times \symm_n$ module (at least when $n! \neq 0$ in $\FFF$). Our calculation of the module structure of $\FFF[\xx_{n \times n}]/I_n$ will make crucial use of a notion of complexity on permutation statistics due to Hamaker and the author \cite{HR} called `locality'.

A {\em permutation statistic} (with values in the field $\FFF$) is a function $f: \symm_n \rightarrow \FFF$.
The study of permutation statistics is an important subfield of combinatorics.
Examples include the {\em exceedance, inversion,} and {\em peak} numbers given by
\begin{align}
\exc(w) &:= | \{ 1 \leq i \leq n \,:\, w(i) > i \} | \\
\inv(w) &:= | \{ 1 \leq i < j \leq n \,:\, w(i) > w(j) \} | \\
\peak(w) &:= | \{ 1 < i < n \,:\, w(i-1) < w(i) > w(i+1) \} |.
\end{align}

Following \cite{HR}, we define a notion of locality for permutation statistics as follows.
 If $\RRR \subseteq [n] \times [n]$ is a rook placement and $w \in \symm_n$ is a permutation, we say that 
 {\em $w$ extends $\RRR$} if we have the containment of sets $\RRR \subseteq \{ (i,w(i)) \,:\, 1 \leq i \leq n \}$.
Given a rook placement $\RRR \subseteq [n] \times [n]$, let $\one_{\RRR}: \symm_n \rightarrow \FFF$ be the indicator permutation statistic 
 \begin{equation}
 \one_{\RRR}(w) = \begin{cases}
 1 & \text{if $w$ extends $\RRR$}, \\
 0 & \text{otherwise,}
 \end{cases}
 \end{equation}
 which detects whether $w$ extends $\RRR$.
 A permutation statistic $f: \symm_n \rightarrow \FFF$ is {\em $k$-local} if there exist field elements $c_{\RRR} \in \FFF$ such that
\begin{equation}
f = \sum_{|\RRR| \, = \, k} c_{\RRR} \cdot \one_{\RRR}
\end{equation}
as functions $\symm_n \rightarrow \FFF$
where the sum is over all rook placements $\RRR \subseteq [n] \times [n]$ with $k$ rooks.

\begin{remark}
A $k$-local statistic $f: \symm_n \rightarrow \FFF$ is also known to have `degree at most $k$' elsewhere in the literature, e.g. \cite{DFLLV}.
We avoid this terminology to guard against confusion with the degree of a character.
\end{remark}

Roughly speaking,
the locality of a permutation statistic bounds its complexity.
The only 0-local statistics are constant functions $\symm_n \rightarrow \FFF$.
The statistic $\exc$ is 1-local, the statistic $\inv$ is 2-local, and the statistic $\peak$ is 3-local.
Following Hamaker and the author \cite{HR}, we consider the $\FFF$-vector space
\begin{equation}
\Loc_k(\symm_n,\FFF) := \{ f: \symm_n \rightarrow \FFF \,:\, f \text{ is $k$-local} \}
\end{equation}
of $k$-local statistics on $\symm_n$.
It is not hard to see that any $k$-local statistic is also $(k+1)$-local, so that 
$\Loc_k(\symm_n,\FFF) \subseteq \Loc_{k+1}(\symm_n,\FFF)$.
Furthermore, any permutation statistic $\symm_n \rightarrow \FFF$ is $(n-1)$-local. The vector spaces $\Loc_k(\symm_n,\FFF)$ will play an important role in the 
module structure of $\FFF[\xx_{n \times n}]/I_n$ (Theorem~\ref{graded-module-structure}); for now we use shadow monomials to solve an open problem from \cite{HR} about 
the spaces $\Loc_k(\symm_n,\FFF)$ themselves.

By definition, the set $\{ \one_{\RRR} \,:\, | \RRR | = k \}$ of indicator statistics corresponding to rook placements $\RRR \subseteq [n] \times [n]$ of size $k$
is a spanning set of  $\Loc_k(\symm_n,\FFF)$, but this spanning set is almost always
linearly dependent.
In \cite[Cor. 4.7]{HR} it is proven that when $\FFF = \mathbb{R}$ is the field of real numbers,
the dimension of $\Loc_k(\symm_n,\FFF)$ equals to the number $a_{n,n-k} + \cdots + a_{n,n-1} + a_{n,n}$ of
permutations in $\symm_n$ which have an increasing subsequence of length at least $n-k$. The methods of \cite{HR} apply whenever $\FFF$ has characteristic 0 or 
characteristic $p > n$; we will see (Theorem~\ref{local-basis-theorem}) that this is true over any field.

The paper \cite{HR} did not give an explicit basis of the space of $k$-local statistics consisting of statistics of the form $\one_\RRR$;
we solve this problem in Theorem~\ref{local-basis-theorem} below.  
Although the members $\one_\RRR$ of our basis for $\Loc_k(\symm_n,\FFF)$ can correspond to rook placements with $|\RRR| < k$ in general,
we will obtain a nested family of bases for the chain of vector spaces
$\Loc_0(\symm_n,\FFF) \subseteq \Loc_1(\symm_n,\FFF) \subseteq \cdots \subseteq \Loc_{n-1}(\symm_n,\FFF)$.
To achieve these goals, we recall a standard fact about associated graded ideals.

Let $\xx$ be a finite set of variables and consider the polynomial ring $\FFF[\xx]$ over these variables.  Given $d \geq 0$ and a graded $\FFF$-algebra $A$,
let $A_{\leq d} \subseteq A$ be the subspace of elements of degree at most $d$.
We have a  filtration $\FFF[\xx]_{\leq 0} \subseteq \FFF[\xx]_{\leq 1} \subseteq \FFF[\xx]_{\leq 2} \subseteq \cdots $ of 
$\FFF[\xx]$ by finite-dimensional subspaces.

\begin{lemma}
\label{filtration-lemma}
Let $I \subseteq \FFF[\xx]$ be an ideal and let $\gr \, I \subseteq \FFF[\xx]$ be the associated graded ideal of $I$.
Fix an integer $d \geq 0$ and let $\BBB \subseteq \FFF[\xx]_{\leq d}$ be a family of homogeneous polynomials of degree at most $d$.
Suppose that $\BBB$ descends to a basis of the vector space $( \FFF[\xx] / \gr \, I )_{\leq d}$.
Then $\BBB$ descends to a basis of the vector space $\FFF[\xx]_{\leq d} / (I \cap \FFF[\xx]_{\leq d} )$.
\end{lemma}

Lemma~\ref{filtration-lemma} is the heart of the orbit harmonics isomorphisms \eqref{long-chain}.
We include its straightforward proof 
for completeness.

\begin{proof}
If $\BBB$ were not linearly independent modulo $I \cap \FFF[\xx]_{\leq d}$, there would exist scalars $c_b \in \FFF$ not all zero and an element $g \in I$ with $\deg(g) \leq d$
such that $\sum_{b \in \BBB} c_b \cdot b = g$.  Since the elements of $\BBB$ are homogeneous, taking the highest degree component of both sides of this equation would result
in a linear dependence of $\BBB$ modulo $\gr \, I$, a contradiction.

If $\BBB$ did not span $\FFF[\xx]_{\leq d} / (I \cap \FFF[\xx]_{\leq d} )$, there would be some homogeneous polynomial $h \in \FFF[\xx]_{\leq d}$ such that 
$g$ does not lie in the span of $\BBB$ modulo $I \cap \FFF[\xx]_{\leq d}$.
Choose such an $h$ with $\deg(h)$ minimal. There exist scalars $c_b \in \FFF$ such that $\sum_{b \in \BBB} c_b \cdot b  = h + \tau(g)$ for some $g \in I$ with 
$\deg(g) = \deg(h)$ (so that in particular $g \in I \cap \FFF[\xx]_{\leq d}$), 
where 
$\tau(g)$ is the highest degree component of $g$. Discarding redundant terms if necessary, we may assume that $c_b = 0$ whenever $\deg(b) \neq \deg(h)$.
We conclude that $h + g - \sum_{b \in \BBB} c_b \cdot b$ has degree $< \deg(h)$, so by our choice of $h$ there exist $c'_b \in \FFF$ and $g' \in I \cap \FFF[\xx]_{\leq d}$
with 
\begin{equation*}
\sum_{b \in \BBB} c'_b \cdot b = h + g - \sum_{b \in \BBB} c_b \cdot b + g'
\end{equation*}
so that $h = \sum_{b \in \BBB} (c'_b - c_b) \cdot b - (g' + g)$ lies in the span of $\BBB$ modulo $I \cap \FFF[\xx]_{\leq d}$, a contradiction.
\end{proof}

An application of Lemma~\ref{filtration-lemma} gives a basis of the vector space $\Loc_k(\symm_n,\FFF)$.

\begin{theorem}
\label{local-basis-theorem}
The vector space $\Loc_k(\symm_n,\FFF)$ of $k$-local statistics $\symm_n \rightarrow \FFF$ has basis
\begin{equation}
\{ \one_{\SSS(w)} \,:\, w \in \symm_n, \, \, \lis(w) \geq n-k \}
\end{equation}
given by indicator functions of shadow sets of permutations in $\symm_n$ which contain an increasing subsequence of length $n-k$.
\end{theorem}

Other authors (see e.g. \cite{DFLLV}) refer to the functions $\one_\RRR$ as {\em juntas}.
So Theorem~\ref{local-basis-theorem} describes a basis of {\em shadow juntas}.

\begin{proof}
For  $\ell \leq k$, any $\ell$-local permutation statistic is also $k$-local, so the indicator functions in question are members of 
$\Loc_k(\symm_n,\FFF)$ by Lemma~\ref{shadow-size-lemma}. Identifying $\symm_n = P_n$ with the locus of permutation matrices in $\FFF^{n \times n}$, the indicator
function $\one_{\RRR}$ corresponding to a rook placement $\RRR \subseteq [n] \times [n]$ is represented by the degree $|\RRR|$
 monomial $m(\RRR) \in \FFF[\xx_{n \times n}]$. It follows that we have an isomorphism
 \begin{equation}
 \label{local-isomorphism}
 \Loc_k(\symm_n,\FFF) \cong \FFF[\xx_{n \times n}]_{\leq k} / ( \II(P_n) \cap \FFF[\xx_{n \times n}]_{\leq k} )
 \end{equation}
 of $\FFF$-vector spaces given by $\one_\RRR \mapsto m(\RRR) + ( \II(P_n) \cap \FFF[\xx_{n \times n}]_{\leq k} )$. Write
 \begin{equation}
  \BBB = \{ \sss(w) \,:\, w \in \symm_n \text{ has an increasing subsequence of length at least $n-k$}  \} \subseteq \FFF[\xx_{n \times n}]
 \end{equation}
 for the set of monomials representing the indicator functions in the statement.
 Theorem~\ref{vector-space-structure} implies that $\BBB$ descends to a basis for the $\FFF$-vector space
 $( \FFF[\xx_{n \times n}] / \gr \, \II(P_n) )_{\leq k}$.  
 An application of
 Lemma~\ref{filtration-lemma} shows that $\BBB$ also descends to a basis for 
 $\FFF[\xx_{n \times n}]_{\leq k} / ( \II(P_n) \cap \FFF[\xx_{n \times n}]_{\leq k} )$, and the isomorphism \eqref{local-isomorphism} completes the proof.
\end{proof}

The nested shadow junta bases of $\Loc_0(\symm_3,\FFF) \subset \Loc_1(\symm_3,\FFF) \subset \Loc_2(\symm_3,\FFF)$ are as follows.

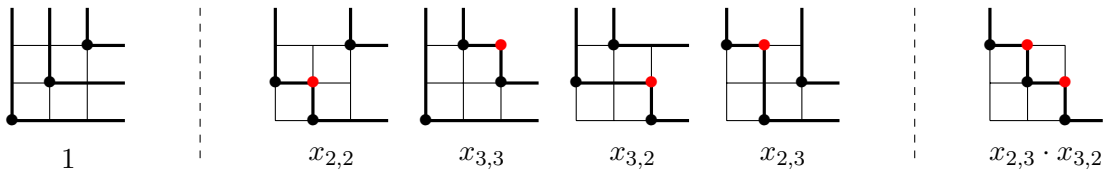
\begin{center}
\begin{tikzpicture}[scale = 0.5]

\draw (1,1) grid (3,3);

\draw[very thick] (1,4) -- (1,1) -- (4,1);
\draw[very thick] (2,4) -- (2,2) -- (4,2);
\draw[very thick] (3,4) -- (3,3) -- (4,3);

\node at (1,1) {$\bullet$};
\node at (2,2) {$\bullet$};
\node at (3,3) {$\bullet$};

\node at (2.5,0) {$1$};

\draw[dashed] (6,4) -- (6,0);

\draw (8,1) grid (10,3);

\draw[very thick] (8,4) -- (8,2) -- (9,2) -- (9,1) -- (11,1);
\draw[very thick] (10,4) -- (10,3) -- (11,3);

\node at (8,2) {$\bullet$};
\node at (9,1) {$\bullet$};
\node at (10,3) {$\bullet$};

\node at (9,2) {${\color{red} \bullet}$};

\node at (9.5,0) {$x_{2,2}$};

\draw (12,1) grid (14,3);

\draw[very thick] (12,4) -- (12,1) -- (15,1);
\draw[very thick] (13,4) -- (13,3) -- (14,3) -- (14,2) -- (15,2);

\node at (12,1) {$\bullet$};
\node at (13,3) {$\bullet$};
\node at (14,2) {$\bullet$};

\node at (14,3) {${\color{red} \bullet}$};

\node at (13.5,0) {$x_{3,3}$};

\draw (16,1) grid (18,3);

\draw[very thick] (16,4) -- (16,2) -- (18,2) -- (18,1) -- (19,1);
\draw[very thick] (17,4) -- (17,3) -- (19,3);

\node at (16,2) {$\bullet$};
\node at (17,3) {$\bullet$};
\node at (18,1) {$\bullet$};

\node at (18,2) {${\color{red} \bullet}$};

\node at (17.5,0) {$x_{3,2}$};

\draw (20,1) grid (22,3);

\draw[very thick] (20,4) -- (20,3) -- (21,3) -- (21,1) -- (23,1);
\draw[very thick] (22,4) -- (22,2) -- (23,2);

\node at (20,3) {$\bullet$};
\node at (21,1) {$\bullet$};
\node at (22,2) {$\bullet$};

\node at (21,3) {${\color{red} \bullet}$};

\node at (21.5,0) {$x_{2,3}$};

\draw[dashed] (25,4) -- (25,0);

\draw (27,1) grid (29,3);

\draw[very thick] (27,4) -- (27,3) -- (28,3) -- (28,2) -- (29,2) -- (29,1) -- (30,1);

\node at (27,3) {$\bullet$};
\node at (28,2) {$\bullet$};
\node at (29,1) {$\bullet$};

\node at (28,3) {${\color{red} \bullet}$};
\node at (29,2) {${\color{red} \bullet}$};

\node at (28.5,0) {$x_{2,3} \cdot x_{3,2}$};

\end{tikzpicture}
\end{center}
It may be interesting to find a basis of $\Loc_k(\symm_n,\FFF)$ drawn from the spanning set $\{ \one_\RRR \,:\, |\RRR| = k \}$.
By Theorem~\ref{vector-space-structure}, the above monomials also form a vector space basis of 
$\FFF[\xx_{3 \times 3}]/I_3$.

The results we have proven so far hold when the field $\FFF$ is replaced by a commutative ring $R$.
More precisely, we have an ideal $I_n^R \subseteq R[\xx_{n \times n}]$ with the same generating set as in 
Definition~\ref{i-defn}.
\begin{itemize}
\item
The proofs of
Lemmas \ref{injective-lemma} and \ref{spanning-lemma} goes through to show that the shadow monomials $\{ \sss(w) \,:\, w \in \symm_n \}$ span 
$R[\xx_{n \times n}]/I_n^R$ over $R$. Here we use the fact that the coefficients in the polynomials $a_{S,T}, b_{S,T}$ appearing in Lemma~\ref{injective-lemma}
are all $\pm 1$.
\item When $R = \ZZ$,
a linear dependence of $\{ \sss(w) \,:\, w \in \symm_n \}$ modulo $I_n^\ZZ$ would induce a linear dependence modulo 
$I_n^\QQ$. By Theorem~\ref{vector-space-structure}
$\{ \sss(w) \,:\, w \in \symm_n \}$ descends to a $\ZZ$-basis of 
$\ZZ[\xx_{n \times n}]/I_n^\ZZ$.  
\item
Since $R[\xx_{n \times n}]/I_n^R = R \otimes_\ZZ \ZZ[\xx_{n \times n}]/I_n^\ZZ$, the set 
$\{ \sss(w) \,:\, w \in \symm_n \}$ descends to a $R$-basis of $R[\xx_{n \times n}]/I_n^R$ for any $R$.
The proof of
Lemma~\ref{filtration-lemma} holds over $R$, so the shadow juntas  $\{ \one_{\SSS(w)} \,:\, w \in \symm_n, \, \, \lis(w) \geq n-k \}$ form an $R$-basis
of $\Loc_k(\symm_n,R)$.
\end{itemize}

\section{Module structure}
\label{Module}

As explained in the introduction,
the self-product $\symm_n \times \symm_n$ of the rank $n$ symmetric group acts on the matrix $\xx_{n \times n}$ 
of variables by independent row and column permutation. 
This induces an action of $\symm_n \times \symm_n$ on $I_n$, and endows 
$\FFF[\xx_{n \times n}]/I_n$ with the structure of a graded $\symm_n \times \symm_n$-module. The purpose of this section is to study this action.
To do so, for the remainder of the section we made the following assumption on the characteristic of $\FFF$.
\begin{quote}
{\bf Assumption.}
{\em The field $\FFF$ either has characteristic zero, or has characteristic $p > n$.}
\end{quote}
This assumption guarantees that the group algebras $\FFF[\symm_n]$ and $\FFF[\symm_n \times \symm_n]$ are semisimple.
We may immediately describe the ungraded $\symm_n \times \symm_n$-structure of $\FFF[\xx_{n \times n}]/I_n$.

\begin{corollary}
\label{ungraded-module-structure}
Let $\symm_n \times \symm_n$ act on the locus $P_n \subseteq \FFF^{n \times n}$  by independent row and column permutation.
We have an isomorphism $\FFF[P_n] \cong \FFF[\xx_{n \times n}] / I_n$ of ungraded $\symm_n \times \symm_n$-modules.
\end{corollary}

Corollary~\ref{ungraded-module-structure} may be given as a decomposition into $\symm_n \times \symm_n$ irreducibles as follows.
If $\lambda \vdash n$ is a partition of $n$, recall that $V^{\lambda}$ denotes the corresponding irreducible $\symm_n$-module.
Irreducible representations of the product group $\symm_n \times \symm_n$ are given by tensor products
$V^{\lambda} \otimes V^{\mu}$ for ordered pairs of partitions $(\lambda, \mu)$ of $n$. Corollary~\ref{ungraded-module-structure} asserts that 
$\FFF[\xx_{n \times n}] / I_n \cong \bigoplus_{\lambda \vdash n} V^\lambda \otimes V^\lambda$ as ungraded $\symm_n \times \symm_n$-modules. 

\begin{proof}
By Theorem~\ref{vector-space-structure} we have an isomorphism and an equality
\begin{equation}
\FFF[P_n] \cong \FFF[\xx_{n \times n}] / \gr \, \II(P_n) = \FFF[\xx_{n \times n}] / I_n
\end{equation}
of ungraded $\FFF$-vector spaces. By our assumption on the characteristic of $\FFF$, these upgrade to an isomorphism and an equality 
of ungraded $\FFF[\symm_n \times \symm_n]$-modules.
\end{proof}

We enhance Corollary~\ref{ungraded-module-structure} by describing the graded module structure of 
$\FFF[\xx_{n \times n}]/I_n$. As suggested by Corollary~\ref{hilbert-series}, the graded refinement of the isomorphism $\FFF[\xx_{n \times n}] / I_n \cong_{\symm_n \times \symm_n} \bigoplus_{\lambda \vdash n} V^\lambda \otimes V^\lambda$ is obtained by focusing on the length of the first row of $\lambda$.

\begin{theorem}
\label{graded-module-structure}
For any $k \geq 0$, the degree $k$ piece of $\FFF[\xx_{n \times n}] / I_n$ has $\symm_n \times \symm_n$-module structure
\begin{equation}
( \FFF[\xx_{n \times n}] / I_n )_k \cong \bigoplus_{\substack{\lambda \, \vdash \, n \\ \lambda_1  \, = \, n-k}} V^{\lambda} \otimes V^{\lambda}.
\end{equation}
\end{theorem}

\begin{proof}
If $W$ is any $\symm_n$-module over $\FFF$, the vector space $\End_\FFF(W)$ of $\FFF$-linear maps $\varphi: W \rightarrow W$ is a
$\symm_n \times \symm_n$-module via 
\begin{equation}
((u,v) \cdot \varphi)(w) := u \cdot \varphi(v^{-1} \cdot w)  \quad \quad \text{for all } u, v \in \symm_n, \, \varphi \in \End_\FFF(W), \, w \in W.
\end{equation}
We have $\End_\FFF(W) \cong W \otimes W^*$ and, since  $\symm_n$-modules are self-dual, we have 
\begin{equation}
\End_\FFF(W) \cong W \otimes W
\end{equation}
as $\symm_n \times \symm_n$-modules.

The group algebra $\FFF[\symm_n]$ is naturally a $\symm_n \times \symm_n$-module under left and right multiplication. Since $\FFF[\symm_n]$ is semisimple,
the {\em Artin-Wedderburn Theorem} gives
an isomorphism of $\FFF$-algebras
\begin{equation}
\label{aw-isomorphism}
\Psi: \FFF[\symm_n] \xrightarrow{ \, \, \sim \, \, } \bigoplus_{\lambda \, \vdash \, n} \End_\FFF(V^{\lambda}).
\end{equation}
Given $a \in \FFF[\symm_n]$, the $\lambda^{th}$ component of $\Psi(a)$ acts on $V^{\lambda}$ by the $\FFF$-linear map $\Psi(a): v \mapsto a \cdot v$.

Returning to the statement of the theorem, since $\FFF[\symm_n]$ is semisimple, by induction on $k$ it suffices to establish the isomorphism
\begin{equation}
( \FFF[\xx_{n \times n}] / I_n )_{\leq k} \cong \bigoplus_{\substack{\lambda \, \vdash \, n \\ \lambda_1  \, \geq \, n-k}} \End_\FFF(V^{\lambda})
\end{equation}
in the category of {\em ungraded} $\symm_n \times \symm_n$-modules.
To this end,
Theorem~\ref{vector-space-structure} gives rise to the identifications
\begin{equation}
\label{big-chain}
\FFF[\symm_n] = \FFF[P_n] \cong \FFF[\xx_{n \times n}] / \II(P_n) \cong \FFF[\xx_{n \times n}] / \gr \, \II(P_n) = \FFF[\xx_{n \times n}] / I_n
\end{equation}
of ungraded $\symm_n \times \symm_n$-modules.
Let $L_k$ be the image of $\FFF[\xx_{n \times n}]_{\leq k}$ in $\FFF[\xx_{n \times n}] / \II(P_n)$, i.e.
\begin{equation}
L_k := \mathrm{Image}( \FFF[\xx_{n \times n}]_{\leq k} \hookrightarrow \FFF[\xx_{n \times n}] \twoheadrightarrow \FFF[\xx_{n \times n}] / \II(P_n) ).
\end{equation}
Lemma~\ref{filtration-lemma} implies that 
\begin{equation}
L_k = 
\spn_\FFF \{ m(\RRR) + \II(P_n)  \,:\, \text{$\RRR$ a rook placement with $|\RRR| \leq k$} \}.
\end{equation}
As explained in the proof of Theorem~\ref{local-basis-theorem}, under the correspondence $\FFF[\symm_n] = \FFF[P_n] \cong \FFF[\xx_{n \times n}]/\II(P_n)$
we have the identification
\begin{equation}
 \Loc_k(\symm_n,\FFF) = L_k
\end{equation}
with the $\symm_n \times \symm_n$-module of $k$-local statistics $\symm_n \rightarrow \FFF$.
Lemma~\ref{filtration-lemma}  and the chain \eqref{big-chain} of isomorphisms give rise to the further identification
\begin{equation}
 \Loc_k(\symm_n,\FFF) = L_k \cong ( \FFF[\xx_{n \times n}] / I_n )_{\leq k}
\end{equation}
of $\symm_n \times \symm_n$-modules.

By the last paragraph, we are reduced to establishing the isomorphism
\begin{equation}
\label{new-target-isomorphism}
\Loc_k(\symm_n,\FFF) \cong \bigoplus_{\substack{\lambda \, \vdash \, n \\ \lambda_1  \, \geq \, n-k}} \End_\FFF(V^{\lambda})
\end{equation}
of ungraded $\symm_n \times \symm_n$-modules.
Embed $\symm_{n-k} \subseteq \symm_n$ by acting on the first $n-k$ letters, let
\begin{equation}
\eta_{n-k} := \sum_{w \in \symm_{n-k}} w \in \FFF[\symm_n]
\end{equation}
be the group algebra element which symmetrizes over these letters, and let $J_k \subseteq \FFF[\symm_n]$ be the two-sided ideal generated by $\eta_{n-k}$.
The correspondence between functions $f: \symm_n \rightarrow \FFF$ and group algebra elements $\sum_{w \in \symm_n} f(w) \cdot w$ gives rise to an identification
\begin{equation}
\label{eq:local-is-j}
\Loc_k(\symm_n,\FFF) = J_k
\end{equation}
of ungraded $\symm_n \times \symm_n$-modules.
Indeed, the group algebra element $\eta_{n-k} \in \FFF[\symm_n]$ corresponds to the indicator permutation statistic $\one_{\RRR}: \symm_n \to \FFF$ indexed by the rook placement $$\RRR = \{ (n-k+1,n-k+1), \dots, (n-1,n-1), (n,n) \}.$$ 
Multiplying $\eta_{n-k}$ on the left and right by permutations of $\symm_n$ corresponds to interchanging rows and columns in the rook placement $\RRR_0$;  any rook placement with $k$ rooks may be obtained in this way.

Thanks to the identification \eqref{eq:local-is-j}, we are reduced to showing
\begin{equation}
\label{brand-newtarget-isomorphism}
J_k \cong \bigoplus_{\substack{\lambda \, \vdash \, n \\ \lambda_1  \, \geq \, n-k}} \End_\FFF(V^{\lambda})
\end{equation}
as $\symm_n \times \symm_n$-modules.
The image $\Psi(J_k)$ of the ideal $J_k \subseteq \FFF[\symm_n]$ under the 
Artin-Wedderburn isomorphism \eqref{aw-isomorphism} is an ideal in the direct sum
$\bigoplus_{\lambda \vdash n} \End_\FFF(V^{\lambda})$ of matrix rings.
Since each summand $\End_\FFF(V^{\lambda})$ is simple, there is a set $P(k)$ of partitions of $n$ such that 
\begin{equation}
\Psi(J_k) = \bigoplus_{\lambda \, \in \, P(k)} \End_\FFF(V^{\lambda}).
\end{equation}
The definitions of $\Psi$ and $J_k$ imply that
\begin{equation}
P(k) = \{ \lambda \vdash n \,:\, \eta_{n-k} \cdot V^{\lambda} \neq 0 \}.
\end{equation}

It remains to show that $P(k) = \{ \lambda \vdash n \,:\, \lambda_1 \geq n-k \}$.  To this end, observe that for any $\symm_n$-module $W$,
the image $\eta_{n-k} \cdot W$ may be characterized as the trivial component
\begin{equation}
\eta_{n-k} \cdot W = \left(  \Res^{\symm_n}_{\symm_{n-k}} W \right)^{\triv}
\end{equation}
of the restriction of $W$ from $\symm_n$ to $\symm_{n-k}$. In particular, for $\lambda \vdash n$ we have
\begin{equation}
\lambda \in P(k) \quad \Leftrightarrow \quad 
\left(  \Res^{\symm_n}_{\symm_{n-k}} V^\lambda \right)^{\triv} \neq 0.
\end{equation}
By the Branching Rule for symmetric group representations (see e.g. \cite[Thm. 2.8.3]{Sagan}),
 the restriction $ \Res^{\symm_n}_{\symm_{n-k}} V^\lambda$
has a nonzero trivial component if and only if $\lambda_1 \geq n-k$.
This proves the isomorphism \eqref{new-target-isomorphism} and the theorem.
\end{proof}

The ring $\FFF[\xx_{n \times n}]$ carries a natural involution $\sigma: x_{i,j} \mapsto x_{j,i}$ which transposes the matrix $\xx_{n \times n}$ of variables.
This induces a homogeneous involution on the quotient ring $\FFF[\xx_{n \times n}]/I_n$, also denoted $\sigma$.
The proof technique of Theorem~\ref{graded-module-structure} applies to show that in the isomorphism
\begin{equation}
( \FFF[\xx_{n \times n}] / I_n )_k \cong \bigoplus_{\substack{\lambda \, \vdash \, n \\ \lambda_1  \, = \, n-k}} V^{\lambda} \otimes V^{\lambda}
\end{equation}
of $\symm_n \times \symm_n$-modules, the action of $\sigma$ on the left hand side intertwines with the automorphism $(w,u) \mapsto (u,w)$
of the group $\symm_n \times \symm_n$.

Recall from the introduction that $\alpha_{n,k}$ is the character of $\symm_n$ given by
$\alpha_{n,k} = \sum_{\lambda_1 = k} f^{\lambda} \cdot \chi^\lambda$, where the sum is over partitions $\lambda \vdash n$
whose first row has length $k$.
As an immediate application of Theorem~\ref{graded-module-structure}, we get an explicit $\symm_n$-module with character $\alpha_{n,k}$.

\begin{corollary}
\label{novak-rhoades-corollary}
The class function $\alpha_{n,k}: \symm_n \rightarrow \FFF$ is the character of the restriction of the degree $n-k$ part of 
$\FFF[\xx_{n \times n}]/I_n$ to either factor of $\symm_n \times \symm_n$.  In symbols, we have
\begin{equation}
\alpha_{n,k} = 
\Res^{\symm_n \times \symm_n}_{\symm_n \times 1} \left(  \chi_{ (\FFF[\xx_{n \times n}]/I_n)_{n-k} }  \right) =
\Res^{\symm_n \times \symm_n}_{1 \times \symm_n} \left(  \chi_{ (\FFF[\xx_{n \times n}]/I_n)_{n-k} }  \right) ,
\end{equation}
where $\chi_V: \symm_n \times \symm_n \rightarrow \FFF$ denotes the character of an $\FFF[\symm_n \times \symm_n]$-module $V$.
\end{corollary}

The space $(\FFF[\xx_{n \times n}]/I_n)_{n-k}$ is the cleanest representation-theoretic model for $\alpha_{n,k}$ known to the author.
There is another model for $\alpha_{n,k}$ involving quotient spaces.
For any $d$, we have an action of $\symm_n$ on $\Loc_d(\symm_n,\FFF)$ given by $(w \cdot f)(v) := f(w^{-1} v)$ for $w, v \in \symm_n$ and
$f \in \Loc_d(\symm_n,\FFF)$.  
The isomorphism \eqref{new-target-isomorphism} implies that 
the sum $\alpha_{n,k} + \alpha_{n,k+1} + \cdots + \alpha_{n,n}$ is the character of $\Loc_{n-k}(\symm_n,\FFF)$.
Therefore, the quotient module $\Loc_{n-k}(\symm_n,\FFF)/\Loc_{n-k-1}(\symm_n,\FFF)$ has character $\alpha_{n,k}$.

Sums of the characters $\alpha_{n,k}$ also arise in the context of Schur-Weyl duality.
Let $\FFF = \CC$, let $V = \CC^d$, and let $V^{\otimes n} = V \otimes \cdots \otimes V$ be the $n$-fold tensor power of $V$.
The vector space $V^{\otimes n}$ carries a diagonal action of $GL(V)$, viz.
\begin{equation}
g \cdot (v_1 \otimes \cdots \otimes v_n) := (g \cdot v_1) \otimes \cdots \otimes (g \cdot v_n) \quad \quad (g \in GL(V), \, \, v_1, \dots, v_n \in V).
\end{equation}
Let $\End_{GL(V)}(V^{\otimes n})$ be the algebra of linear maps $\varphi: V^{\otimes n} \rightarrow V^{\otimes n}$ which commute with the action
of $GL(V)$.  We have an algebra homomorphism $\Phi: \CC[\symm_n] \rightarrow \End_{GL(V)}(V^{\otimes n})$ induced by
\begin{equation}
\Phi(w) \cdot (v_1 \otimes \cdots \otimes v_n) := v_{w^{-1}(1)} \otimes \cdots \otimes v_{w^{-1}(n)} \quad \quad (w \in \symm_n, \, \, v_1, \dots, v_n \in V).
\end{equation}
Schur-Weyl duality asserts that the homomorphism $\Phi$ is surjective, but when $d < n$ the kernel of $\Phi$ is nonzero.
In fact, the character of the $\symm_n$-module $\End_{GL(V)}(V^{\otimes n})$ is given by 
\begin{equation}
\chi_{\End_{GL(V)}(V^{\otimes n})} = \sign \otimes (\alpha_{n,1} + \alpha_{n,2} + \cdots + \alpha_{n,d})
\end{equation}
where $\sign$ is the degree 1 sign character. In other words, we have 
$\chi_{\End_{GL(V)}(V^{\otimes n})} = \sum_{\lambda'_1 \leq d} f^{\lambda} \cdot \chi^{\lambda}$ where the sum is over partitions $\lambda \vdash n$
whose first column has length at most $d$.
By Corollary~\ref{novak-rhoades-corollary}, we have an isomorphism of $\symm_n$-modules
\begin{equation}
\End_{GL(V)}(V^{\otimes n}) \cong_{\symm_n} \sign \otimes \bigoplus_{k \, \geq \, n-d} (\CC[\xx_{n \times n}]/I_n)_k.
\end{equation}
It may be interesting to give a formula for this isomorphism.

 By Corollary~\ref{novak-rhoades-corollary}, finding an explicit family of linear injections
\begin{equation}
(\FFF[\xx_{n \times n}]/I_n)_{d-1} \otimes 
(\FFF[\xx_{n \times n}]/I_n)_{d+1}  \hookrightarrow 
(\FFF[\xx_{n \times n}]/I_n)_{d} \otimes
(\FFF[\xx_{n \times n}]/I_n)_{d} \quad \quad (0 < d < n-1)
\end{equation}
which commute with either the row or column action of $\symm_n$ on $\xx_{n \times n}$ would prove the Novak-Rhoades conjecture \cite{NR} and 
imply Chen's conjecture \cite{Chen}.
In fact, computations suggest that such an injection can be found which commutes with both row and column permutation.

\begin{conjecture}
\label{equivariant-conjecture}
Given any degree $d \geq 0$, let $\symm_n \times \symm_n$ act on $ (\FFF[\xx_{n \times n}]/I_n)_d$ by independent row and column permutation.
For all $0 < d < n-1$ there exists a 
linear injection
\begin{equation*}
\varphi: 
(\FFF[\xx_{n \times n}]/I_n)_{d-1} \otimes 
(\FFF[\xx_{n \times n}]/I_n)_{d+1}  \hookrightarrow 
(\FFF[\xx_{n \times n}]/I_n)_{d} \otimes
(\FFF[\xx_{n \times n}]/I_n)_{d}
 \end{equation*}
 which commutes with the diagonal action of $\symm_n \times \symm_n$ defined  by
 \begin{equation*}
 (w, v) \cdot (f \otimes g) := ((w,v) \cdot f) \otimes ((w,v) \cdot g)
 \end{equation*}
 for $(w,v) \in \symm_n \times \symm_n$ and $f, g \in \FFF[\xx_{n \times n}]/I_n$.
\end{conjecture}

Conjecture~\ref{equivariant-conjecture} would imply both the Novak-Rhoades conjecture \cite{NR} and Chen's conjecture \cite{Chen}.
The existence of a map $\varphi$ as in Conjecture~\ref{equivariant-conjecture} has been checked for $n \leq 15$.

\section{Conclusion}
\label{Conclusion}

This paper established a connection between the algebra of $\FFF[\xx_{n \times n}]/I_n$ and the combinatorics of $\symm_n$.
It may be interesting to find analogous results for other combinatorial structures.
As motivation,
 B\'ona, Lackner, and Sagan \cite{BLS} conjectured that the sequence $(i_{n,1}, \dots, i_{n,k})$ given by 
 \begin{equation}
 i_{n,k} = | \{ w \in \symm_n \,:\, \lis(w) = k, \, \, w^2 = 1 \} |
 \end{equation}
 which counts {\em involutions} in $\symm_n$ with longest increasing subsequence of length $k$ is log-concave.
 Novak and the author made (unpublished) the stronger conjecture \cite{NR} that the sequence 
 $(\iota_{n,1}, \dots, \iota_{n,n})$ of characters
 \begin{equation}
 \iota_{n,k} := \sum_{\substack{\lambda \, \vdash \, n \\ \lambda_1 \,  = \, k }} \chi^{\lambda}
 \end{equation}
 is log-concave with respect to the Kronecker product (where a class function is `non-negative' if it is a genuine character).
On the commutative algebra side, adding the differences $x_{i,j} - x_{j,i}$ to the ideal $I_n \subseteq \FFF[\xx_{n \times n}]$ gives a candidate quotient
ring which could be used to study these conjectures.

 A key tool for understanding the structure of $\FFF[\xx_{n \times n}]/I_n$ was the orbit harmonics method applied to the locus $P_n \subseteq \FFF^{n \times n}$
 of permutation matrices; it was proven that $I_n = \gr \, \II(P_n)$.  It may be interesting to compute $\gr \, \II(M_n)$ for other matrix loci
 $M_n \subseteq \FFF^{n \times n}$. Four suggestions in this direction are as follows.
 \begin{enumerate}
 \item  The set $M_n = \III_n$ of symmetric permutation matrices corresponding to involutions in $\symm_n$. The ideal $\gr \, \II(\III_n)$ could have application
 to the B\'ona-Lackner-Sagan conjecture \cite{BLS} and the Kronecker log-concavity of the character sequence $(\iota_{n,1}, \dots , \iota_{n,n})$.
 \item  The set $M_n = G$ of elements of a complex reflection group. The Hilbert series of 
 $\FFF[\xx_{n \times n}]/\gr \, \II(G)$ should be generating functions for a `longest increasing subsequence' statistic on $G$.\footnote{While this paper was under review, M. J. Liu \cite{Liu} solved this problem for the wreath product groups $G = \ZZ_r \wr \symm_n$.}
 \item The set $M_n = A_n$ of $n \times n$ alternating sign matrices. A standard monomial basis of 
  $\FFF[\xx_{n \times n}]/\gr \, \II(A_n)$ could give a clues about a Schensted correspondence for ASMs.
 \end{enumerate}
It may also be interesting to consider loci of rectangular $m \times n$ matrices for which $m \neq n$. For example, fixing sequences $\lambda = (\lambda_1, \dots, \lambda_n)$ and $\mu = (\mu_1, \dots, \mu_n)$, one could consider the {\em contingency table} locus of $\ZZ_{\geq 0}$-matrices with column sums $\lambda$ and row sums $\mu$. Fulton's {\em matrix-ball construction} \cite{Fulton} generalizes  Viennot shadow lines from permutation matrices to contingency tables; perhaps the matrix-ball construction is also related to standard monomial theory.

The genesis of this paper was an email from Pierre Briaud and
Morten {\O}ygarden to the author regarding a problem in cryptography.
We close by describing this problem and its relationship to our work.

Let $q$ be a prime power and let $\FFF_q$ be the finite field with $q$ elements. Given a known matrix $\AAA \in \FFF_q^{m \times n}$ 
and a known vector $\vv \in \FFF_q^n$, the 
{\em Permuted Kernel Problem} \cite[Def. 1]{BFKMPP} seeks to recover an unknown permutation $w \in \symm_n$
of the coordinates of $\vv$ which lies in the right kernel of $\AAA$.
The parameters $q, m,$ and $n$ are chosen so that $n! \approx q^m$ and there exists a unique such $w \in \symm_n$ with high probability.
The PKP amounts to solving a polynomial system in the $n^2$ variables $\xx_{n \times n}$ over the field $\FFF_q$ consisting of 
\begin{enumerate}
\item the polynomials which express $\xx_{n \times n}$ as a permutation matrix, and
\item the $m$ polynomials coming from the vector equation $\AAA  \cdot \xx_{n \times n} \cdot \vv = \zero$.
\end{enumerate}
In cryptography, one wants to know the difficulty in solving this system using  Gr\"obner methods.\footnote{In cryptography, one also often works 
over fields of low characteristic, including characteristic 2. This is one reason why we remained as agnostic as possible  about our choice of field.}
This paper analyzed the system of polynomials coming from $(1)$ alone; we hope that this will lead to a better understanding of the more cryptographically relevant
system
$(1) \cup (2)$.
The Hilbert series of a quotient similar to that by $(1) \cup (2)$ was studied by Briaud and {\O}ygarden in \cite{BO} when the linear system analogous to (2)
is sufficiently generic.

\section{Acknowledgements}

The author is very thankful to Pierre Briaud and Morten {\O}ygarden for sharing their conjecture on the Hilbert series of 
$\FFF[\xx_{n \times n}] / \gr \, \II(P_n)$ and for many very helpful email interactions throughout this project
(and their comments on and corrections to earlier versions of this manuscript).
The author is also grateful to Dennis Stanton and Jessica Striker for helpful conversations.
The author thanks four anonymous referees for their careful reading of and helpful comments on this paper.
The author was partially supported by NSF Grants DMS-1953781 and DMS-2246846.

\end{document}